\documentclass[11pt, a4paper]{article}
\usepackage[english]{babel}
\usepackage{amssymb,amsmath,amsthm,bm,graphicx}
\usepackage{authblk}
\usepackage[authoryear]{natbib}

\newtheorem{theorem}{Theorem}
\newtheorem{lemma}{Lemma}

\theoremstyle{definition}

\DeclareMathOperator{\tr}{tr}

\def\beqn{\begin{eqnarray*}}
\def\eeqn{\end{eqnarray*}}
\def\beq{\begin{eqnarray}}
\def\eeq{\end{eqnarray}}
\def\bm#1{\mbox{\boldmath{$#1$}}}

\newcommand{\diag}{{\text{diag}}}

\begin{document}

\title{\Large\bfseries Adaptive empirical Bayesian smoothing splines}
\author{\bfseries Paulo Serra and Tatyana Krivobokova\\ \vspace{-0.75em} {\rm Georg-August-Universit\"at G\"ottingen}}
\date{\today}

\maketitle

\begin{abstract}
In this paper we develop and study adaptive empirical Bayesian smoothing splines.
These are smoothing splines with both smoothing parameter and penalty order determined via the empirical Bayes method from the marginal likelihood of the model.
The selected order and smoothing parameter are used to construct adaptive credible sets with good frequentist coverage for the underlying regression function.
We use these credible sets as a proxy to show the superior performance of
adaptive empirical Bayesian smoothing splines compared to frequentist smoothing splines.
\end{abstract}


\section[Introduction]{Introduction}\label{sec:introduction}%

Consider $n$ observations from the non-parametric regression model
\begin{equation}\label{eq:model}
Y_i = f(x_i) + \sigma \epsilon_i, \quad i=1,\dots, n.
\end{equation}
The function $f$ is assumed to belong to a Sobolev class $\mathcal{W}_\beta(M)$, a collection of continuous functions $f\in L_2$ such that $f^{(\beta-1)}$ is absolutely continuous and $\|f^{(\beta)}\|^2<M^2$, where $\|\cdot\|$ is the $\ell_2$-norm.
The design points $\bm{x}=(x_1,\dots,x_n)\in[0,1]^n$ are $x_i=(2i-1)/(2n)$, the observation errors $\epsilon_1, \dots, \epsilon_n$ are assumed to be i.i.d.\ standard Gaussian random variables and $\sigma^2>0$.
Parameters $f$, $\beta$, and $\sigma^2$ are unknown and of interest.

In this paper we study a smoothing spline estimator for $f$, which is
the unique minimiser in $\mathcal{W}_q$ of the penalised least squares criterion
\begin{equation}\label{eq:criterion}
\frac1n\sum_{i=1}^n\big\{Y_i - f(x_i)\big\}^2 + \lambda\int_0^1 \big\{f^{(q)}(t)\big\}^2\,dt,\;\;\lambda>0,\;\;q\in\mathbb{N}
\end{equation}
and is well known to be a natural polynomial smoothing spline of
degree $2q-1$ with knots at the observation points; see \cite{Wahba:1990}.

The performance of smoothing splines as data-smoothers crucially
depends on the choice of the smoothing parameter $\lambda$, which
balances fidelity to the data and smoothness of the estimator.  
Criteria to select such a smoothing parameter can be obtained under two paradigms, which correspond to making different assumptions on the data-generating mechanism.
One possibility is to assume that the regression function $f$ is some fixed function from a certain class (frequentist model).
In this case $\lambda$ is estimated by minimising an unbiased
estimator of the mean integrated squared error (unbiased risk estimator).
Generalised cross validation (GCV), Mallow's $C_p$ and Akaike's
information criterion are all asymptotically equivalent criteria of
this type. In the following $\hat\lambda_f$ denotes a minimiser of
one of these criteria.

Another possibility is to assume that the regression function $f$ is a realisation of some stochastic process (Bayesian model).
Here a conjugate prior is put on the regression function,  
 such that the
resulting posterior mean coincides with the smoothing spline estimator.
The smoothing parameter $\lambda$ is then a so-called
\emph{hyper-parameter} of the prior. Its estimator is set to a maximiser of the
resulting marginal likelihood (empirical Bayes method) and will be
denoted by $\hat\lambda_q$. 

The prior we use is a 
Gaussian prior.
For these, conjugacy properties are often explored to directly study the posterior and specific Bayes estimators.
Properties of Gaussian process priors (not necessarily conjugate) can be found in~\cite{Vaart:2008a};
modifications to obtain adaptive priors were proposed in~\cite{Vaart:2009}.
A (by no means extensive) list of results on adaptation using Gaussian priors in regression and the closely related Gaussian white noise model include~\cite{Belitser:2003,Belitser:2010,Knapik:2011,Knapik:2013,deJonge:2010,deJonge:2012,Szabo:2013}.

Bayesian smoothing splines with Gaussian priors have been first considered in
\cite{Kimeldorf:1970}. Extensions and modifications of these splines
have been discussed e.g. in \cite{Kohn:1987}, \cite{Speckman:2003} or \cite{Yue:2014}.

The asymptotic distributions of $\hat\lambda_f$ and $\hat\lambda_q$ can be computed under the assumption that the data come from the frequentist model with $f$ as a fixed, ``true'' regression function of interest.
This allows a direct comparison of these two estimators obtained under different paradigms.
\cite{Krivobokova:2013} shows that $\hat\lambda_f$ and $\hat\lambda_q$
are consistent for certain oracles and that 
the asymptotic variance of $\hat\lambda_f$ can be several times larger than that of $\hat\lambda_q$.

The literature on adaptive Bayesian nonparametric estimation in non-parametric regression, and their frequentist performance is already quite extensive.
For general priors, sufficient conditions for so called posterior contraction were proposed in~\cite{Ghosal:2000,Shen:2001,Ghosal:2007} -- 
posterior contraction at a given rate ensures the existence of frequentist estimators with the same rate.
Adaptation is usually achieved by considering a family of priors indexed by a hyper-parameter (like $\lambda$ and $q$ above).
If the regression function $f$ belongs to a given smoothness class and
there is some value of the hyper-parameter such that the resulting posterior contracts at the minimax rate for $f$ in that class, then either endowing the hyper-parameter with a prior (hierarchical Bayes), or picking the hyper-parameter in a data-driven way (empirical Bayes), can lead to posteriors that contract adaptively.
This approach has be used in~\cite{Belitser:2003b,Zhang:2005,Johnstone:2005,McAuliffe:2006,Ghosal:2008,Knapik:2013,Shen:2015}, among others.
Here we consider the empirical Bayes approach.

Smoothing parameters that minimise an unbiased risk
estimator (e.g., GCV) are predominant in practice and known to have good theoretical properties.
In particular, if there is a mismatch between the order of the spline ($2q-1$) and
the smoothness of the regression function ($f$ admits more than just
$q$ square integrable derivatives), then $\hat\lambda_f$ adapts to
this extra smoothness, but only up to $2q$.
In contrast, $\hat\lambda_q$ does not adapt and its rate is determined
by $q$ only.
The main question that we address in this paper is whether one can
obtain a \emph{selector} $\hat q$, such that the resulting
$\hat\lambda_{\hat q}$ not only adapts to the underlying smoothness of $f$,
but also outperforms $\hat\lambda_f$ due to a much smaller variance.

To derive such a selection criterion for $q$ we use the fact that the
prior distribution depends on $q$, albeit in an implicit way, and apply the empirical Bayes approach.
Contrary to the selection of the smoothing parameter $\lambda$, the
selection of the order of smoothing splines has received very little attention in the literature.

Since $\hat\lambda_f$ and $\hat\lambda_{\hat q}$ are associated with splines of different order, direct comparison between the two smoothing parameters is not adequate.
Instead, we construct credible $\ell_2$ balls with good frequentist
coverage, obtained from a high probability region of the posterior
corresponding to hyper-parameters $\hat q$ and $\hat\lambda_{\hat q}$
selected via empirical Bayes.
Subsequently, we show that if the center of this ball is replaced by a
smoothing spline with the smoothing parameter $\hat\lambda_f$, then the
coverage property is lost, proving superiority of adaptive
empirical Bayesian smoothing splines.

This paper is structured as follows.
In Section~\ref{sec:empirical_bayes} we describe the empirical Bayes approach and define some estimators.
The asymptotic behaviour of the estimators is described in Section~\ref{sec:asymptotics}.
In Section~\ref{sec:coverage} we establish frequentist properties of a specific type of Bayesian credible set.
In Section~\ref{sec:GCV} we compare our approach for the selection of the smoothing parameter with generalised cross validation.
Some numerical experiments can be found in Section~\ref{sec:simulations}.
Section~\ref{sec:conclusions} contains some conclusions.
We refer the reader to Section~\ref{appendix:technical_results} for technical details and proofs.\\

\section[Empirical Bayesian smoothing splines]{Empirical Bayesian smoothing splines}\label{sec:empirical_bayes}%

Let us denote the minimiser of (\ref{eq:criterion}) by
$\hat{f}_{\lambda,q}$, which is a natural smoothing spline of degree
$2q-1$ with knots at the observation points. This smoothing spline estimator is a linear
estimator that satisfies for each $\lambda$ and $q$ 
\begin{equation}\label{eq:smoothing_spline_at_design}
\bm{\hat{f}}_{\lambda,q} = \hat{f}_{\lambda,q}(\bm{x}) = \bm{S}_{\lambda,q}\bm{Y}.
\end{equation}
The positive-definite smoother matrix
$\bm{S}_{\lambda,q}\in\mathbb{R}^{n\times n}$ is known explicitly and
the vector $\bm{Y} = (Y_1,\dots,Y_n)^T$ comprises of the observations from model~\eqref{eq:model}.

The $L_2$-risk of the smoothing
spline estimator~\eqref{eq:smoothing_spline_at_design} of
$f\in\mathcal{W}_\beta$ is a function of $\lambda$ and $q$ known asymptotically; see~\cite{Wahba:1990}.
In particular, with a suitable $\lambda$ that minimises that risk, it
holds for each $q$
\begin{equation}\label{eq:risk}
\mathbb{E}\|\bm{\hat f}_{\lambda,q} - \bm{f}\|^2 \asymp n^{-\frac{2\min(\beta,q)}{2\min(\beta,q)+1}}.
\end{equation}
This is of the order of the minimax risk for estimating $f\in\mathcal{W}_{\min(\beta,q)}$ in model~\eqref{eq:model}.
Apparently, to minimise this risk $q$ should be larger
than the maximum smoothness of $f$. This fact is known in an wider context
of Tikhonov regularisation, see \cite{Lukas:1988} and references
therein. However, we are not aware of any practical methods to select
the optimal penalisation order $q$.  
In this work, estimates for both $\lambda$ and $q$ are obtained via the empirical Bayes method.
We start by specifying a prior on the regression function $f$ and on $\sigma^2$.

Given $(\bm{x},\lambda,q)$ we endow $\sigma^2$ with an inverse-gamma prior with shape parameter $a$, and scale parameter $b$ (both left unspecified for now), and given $(\sigma^2,\bm{x},\lambda,q)$ we endow $\bm{f}$ with a so called partially informative Gaussian prior with mean vector $\bm{0}$ and precision matrix $(\bm{S}_{\lambda,q}^{-1}-\bm{I}_n)/\sigma^2$, which we denote $PN\{\bm{0}, (\bm{S}_{\lambda,q}^{-1}-\bm{I}_n)/\sigma^2\}$ and admits a density which is proportional to
\begin{equation}\label{eq:PIN_dens}
\Big|\frac{\bm{S}_{\lambda,q}^{-1}-\bm{I}_n}{\sigma^2}\Big|^{1/2}_+
\exp\bigg\{-\frac1{2\sigma^2}\bm{f}^T(\bm{S}_{\lambda,q}^{-1}-\bm{I}_n)\bm{f}\bigg\},
\end{equation}
where $|\cdot|_+$ represents the product of the non-zero eigenvalues 
(the smoother matrix $\bm{S}_{\lambda,q}$ has exactly $q$ eigenvalues equal to $1$; cf.~\citealt{Speckman:1985} and~\eqref{eq:diagonalisation} in Section~\ref{appendix:technical_results}).
This prior has two parts, a constant, non-informative prior on the null space of $\bm{S}_{\lambda,q}^{-1}-\bm{I}_n$, and a proper, degenerate Gaussian prior on the range of $\bm{S}_{\lambda,q}^{-1}-\bm{I}_n$; cf.~\citealp{Speckman:2003}.

We say that the prior on $(\bm{f},\sigma^2)$ is a partially-informative-Gaussian-inverse-gamma distribution, and denote it by
\begin{equation}\label{eq:prior}
\Pi_{\lambda,q}(\,\cdot\mid \bm{x}) = PNIG\Big\{\bm{0},\; \bm{S}_{\lambda,q}^{-1}-\bm{I}_n,\; a,\; b\Big\}.
\end{equation}
This prior is conjugate for model~\eqref{eq:model} and the corresponding posterior distribution is
\begin{equation}\label{eq:posterior}
\Pi_{\lambda,q}(\,\cdot\mid \bm{Y},\bm{x}) =
PNIG\Big\{\bm{S}_{\lambda,q}\bm{Y},\; \bm{S}_{\lambda,q}^{-1},\; a+\frac n2,\; b+\frac12\bm{Y}^T(\bm{I}_n-\bm{S}_{\lambda,q})\bm{Y}\Big\}.
\end{equation}

The posterior mean of~\eqref{eq:posterior} is $\bm{\hat{f}}_{\lambda,q}$, and the mean of the predictive posterior distribution can be shown to be the smoothing spline $\hat{f}_{\lambda,q}$.
The prior~\eqref{eq:prior} is improper but the corresponding posterior~\eqref{eq:posterior} is proper.
By definition, the marginal posterior for $\sigma^2$ is an inverse-gamma distribution
\begin{equation}\label{eq:marginal_posterior_sigma}
\Pi_{\lambda,q}^{\sigma^2}(\,\cdot\,|\bm{Y},\bm{x}) =
\int\Pi_{\lambda,q}(d\bm{f},\,\cdot\,|\bm{Y},\bm{x}) =
IG\Big\{a+\frac n2,\; b+\frac12\bm{Y}^T(\bm{I}_n-\bm{S}_{\lambda,q})\bm{Y}\Big\},
\end{equation}
and the marginal posterior for $\bm{f}$ is a non-central, $n$-variate
t-distribution (cf.~\citealp{Kotz:2004})
\begin{equation}\label{eq:marginal_posterior_f}
\begin{aligned}
\Pi_{\lambda,q}^{\bm{f}}(\,\cdot\,|\bm{Y},\bm{x}) =
\int\Pi_{\lambda,q}(\,\cdot\,,d\sigma^2|\bm{Y},\bm{x}) =
t_{2a+n-q}\Big\{\bm{S}_{\lambda,q}\bm{Y},\, \frac{2b+\bm{Y}^T(\bm{I}_n-\bm{S}_{\lambda,q})\bm{Y}}{2a+n-q}\bm{S}_{\lambda,q}\Big\}.
\end{aligned}
\end{equation}

The posterior distribution
$\Pi_{\lambda,q}\big(\,\cdot\mid\bm{Y},\bm{x}\big)$ depends on
$\lambda$ and $q$, and on the (hyper-) parameters $a$ and $b$.
We select the unknown parameters $\lambda$ and $q$ in a data-driven way via estimators $\hat\lambda$ and $\hat q$ and plug them into the posterior~\eqref{eq:posterior} resulting in a new random measure, the empirical Bayes posterior, which is defined as
\begin{equation}\label{eq:empirical_Bayes_posterior}
\Pi_{\hat\lambda,\hat q}\big(\,\cdot\mid\bm{Y},\bm{x}\big) =
\Pi_{\lambda,q}\big(\,\cdot\mid\bm{Y},\bm{x}\big)\big|_{(\lambda,q)=(\hat\lambda,\hat q)}.
\end{equation}
Following~\cite{Robbins:1955}, the empirical Bayes method consists of setting $\lambda$ and $q$ to maximisers of the marginal likelihood of model~\eqref{eq:model} under the Bayesian paradigm.
We designate the mean of the empirical Bayes marginal
posterior~\eqref{eq:empirical_Bayes_posterior} as the \emph{adaptive empirical Bayesian smoothing spline}.

Since the data $\bm{Y}|(f,\sigma^2,\bm{x})$ are distributed like a $N(\bm{f},\; \sigma^2\bm{I}_n)$ random vector, and we endow $\bm{f}|(\sigma^2,\bm{x},\lambda,q)$ with a $PN\big\{\bm{0},\,(\bm{S}_{\lambda,q}^{-1}-\bm{I}_n)/\sigma^2\big\}$ prior, then $\bm{Y}|(\sigma^2,\bm{x},\lambda,q)$ is distributed like a $PN\big\{\bm{0},\; (\bm{I}_n-\bm{S}_{\lambda,q})/\sigma^2\big\}$ random vector.
The variance $\sigma^2|(\bm{x},\lambda,q)$ is endowed with an $IG\big(a,b\big)$ prior, such that $(\bm{Y},\sigma^2)|(\bm{x},\lambda,q)$ is by definition jointly distributed like a $PNIG\big\{ \bm{0},\, \bm{I}_n-\bm{S}_{\lambda,q},\, a,\, b\big\}$ random vector.
By integrating out $\sigma^2$, $\bm{Y}|(\bm{x},\lambda,q)$ is distributed like a $t_{2a-q}\big\{\bm{0},\, 2b(\bm{I}_n-\bm{S}_{\lambda,q})^{-}/(2a-q)\big\}$ random vector, where the superscript ``--'' indicates the pseudo-inverse.
It can be shown that this distribution admits a density with respect to an appropriate dominating measure, resulting, up to some constant $h_n\big(a,b\big)$, in the following marginal log-likelihood for $(\lambda,q)$,
\begin{gather*}
\ell_n(\lambda,q\mid a,b) =
-\Big(a+\frac {n-q}2\Big) \log\Big\{ \bm{Y}^T\big(\bm{I}_n - \bm{S}_{\lambda,q}\big)\bm{Y}  +2b\Big\} + \frac12\log|\bm{I}_n - \bm{S}_{\lambda,q}|_+.
\end{gather*}

The hyper-parameters $a$ and $b$ do not play an important role in our approach so we set $a=q/2$, $b=0$ with the convention that $0^0=1$ (this corresponds to placing an improper prior on $\sigma^2$).
This does simplify the expressions that follow, but $a$ and $b$ can be set to any non-negative value that is $o(n)$ and does not depend on $\lambda$ or $q$, without affecting our results.
We obtain, up to a constant, the marginal log-likelihood
\begin{equation}\label{eq:log_likelihood}
\ell_n(\lambda,q) =
\ell_n(\lambda,q\mid q/2,0) =
-\frac n2 \log\Big\{ \bm{Y}^T\big(\bm{I}_n - \bm{S}_{\lambda,q}\big)\bm{Y}\Big\} + \frac12\log|\bm{I}_n - \bm{S}_{\lambda,q}|_+,
\end{equation}
where $(\lambda,q)$ lives on $(0,\infty)\times\mathbb{N}$; note that $h_n(q/2,0)$ does not depend on $\lambda$ or $q$.

The dependence of \eqref{eq:log_likelihood} on $q$ is
rather implicit, so that it is convenient to represent the smoother
matrix as $\bm{S}_{\lambda,q}=\bm{\Phi}\left\{\bm{I}_n+\lambda
n\diag(\bm{\eta}_{q})\right\}^{-1}\bm{\Phi}^T$. Here
$\bm\Phi$ is the Demmler-Reinsch basis
matrix, such that
$\bm\Phi^T\bm\Phi=\bm\Phi\bm\Phi^T=\bm{I}_n$, and 
$\bm{\eta}_q=(\eta_{q,1},\ldots,\eta_{q,n})^T$, see
 Section~\ref{appendix:DR_basis} 
 for details. In particular, \eqref{eq:diagonalisation} in Section~\ref{appendix:technical_results} gives an
 approximation of $\eta_{q,i}$ as a function of $q$. With
 this, we can re-express~\eqref{eq:log_likelihood} as
\begin{equation}\label{eq:log_likelihood_eigen_representation}
\ell_n(\lambda,q) =
-\frac n2 \log\bigg( \sum_{i=q+1}^n \frac{X_i^2 \lambda n\eta_{q,i}}{1+\lambda n\eta_{q,i}}\bigg) +
\frac12\sum_{i=q+1}^n\log \frac{ \lambda n\eta_{q,i}}{1+\lambda n\eta_{q,i}},
\end{equation}
where $\bm{X} = (X_1,\ldots,X_n)=\bm{\Phi}^T\bm{Y}$. 
Further, based on the approximations from~\eqref{eq:diagonalisation}, $\ell_n(\lambda,q)$ is continuously differentiable.

A maximiser $(\hat\lambda,\hat q)$ of $\ell_n(\lambda,q)$ is found by means of estimating equations, as zeroes of appropriately rescaled partial derivatives of $\ell_n(\lambda,q)$.
Our estimating equations for $\lambda$ and $q$ are respectively
\begin{equation}\label{eq:criteria}
\begin{aligned}
T_\lambda(\lambda,q) &=
-\frac{2\lambda}{n^2}\Big\{\bm{Y}^T(\bm{I}_n-\bm{S}_{\lambda,q})\bm{Y}\Big\}
\frac{\partial\ell_n(q, \lambda)}{\partial\lambda}\\&= 
\frac{1}{n}\sum_{i=q+1}^n \frac{X_i^2 \lambda n\eta_{q,i}}{(1+\lambda n\eta_{q,i})^2}-\frac{1}{n^2} \sum_{i=q+1}^n \frac{ X_i^2\lambda n\eta_{q,i}}{1+\lambda n\eta_{q,i}}\sum_{i=q+1}^n \frac{1}{1+\lambda n\eta_{q,i}}, \quad\text{and}\\
T_q(\lambda,q)	&=
-\frac{2q}{n^2}\Big\{\bm{Y}^T(\bm{I}_n-\bm{S}_{\lambda,q})\bm{Y}\Big\}
\frac{\partial\ell_n(q, \lambda)}{\partial q}\\&=
\frac{1}{n}\sum_{i=q+1}^n\frac{X_i^2 \lambda
  n\eta_{q,i}\log( n\eta_{q,i})}{(1+\lambda n\eta_{q,i})^2}-\frac{1}{n^2} \sum_{i=q+1}^n
\frac{ X_i^2\lambda n\eta_{q,i}}{1+\lambda n\eta_{q,i}}\sum_{i=q+1}^n
\frac{\log( n\eta_{q,i})}{(1+\lambda n\eta_{q,i})}
\\&+
\frac{1}{n}\sum_{i=q+1}^n\frac{\partial X_i^2}{\partial q}\frac{\lambda
   n\eta_i}{1+\lambda n \eta_i},
\end{aligned}
\end{equation}
See Section~\ref{appendix:DR_basis} for details on the
derivation of these expressions. In particular, it is shown that the
contribution of the last term in $T_q(\lambda,q)$ with $\partial
X_i^2/\partial q$ is negligible. 
Note that since $\ell_n(\lambda,q)$ is continuously differentiable, if $\hat\lambda_q$ solves $T_\lambda(\lambda,q)=0$ for each $q$ and $\hat q$ solves $T_q(\hat\lambda_q,q)=0$, then $(\hat\lambda,\hat q)=(\hat\lambda_{\hat q}, \hat q)$.
Note that for each $q$, $\hat\lambda_q$ is essentially the generalised maximum likelihood estimator from~\cite{Wahba:1985}; cf.~(1.5) in~\cite{Wahba:1985} and the criterion $T_\lambda(\lambda,q)$ above.\\

We briefly address some practical issues involving the optimisation of
the criteria in~\eqref{eq:criteria}.
The estimate $\hat\lambda$ is taken on $[1/n,1]$. Parameter $q$ enters the 
eigenvalues $\eta_{q,i}$ according to \eqref{eq:diagonalisation} and $X_i$ as the degree
$2q-1$ of basis
$\bm{\Phi}$. While values $\eta_{q,i}$ are defined for each $q$, the
degree of a spline is, in practice, typically an integer. One practical way to
proceed in minimisation of \eqref{eq:criteria} would be to restrict
$q\in\mathbb{N}$. Alternatively, one could generalise splines to a
fractional order (cf.~\citealp{Unser:2000}, for a representation of
fractional splines in terms of fractional B-splines), which we do not
pursue. Instead, we can use the fact that the contribution of the term
with $\partial
X_i^2/\partial q$  is negligible
(see Section~\ref{appendix:DR_basis} for details). Therefore, we suggest to relax
$q\in(1/2,\log(n)]$ to be real-valued and in practice set
$X_{q,i}=X_{\lfloor q\rfloor,i}$, which allows to estimate 
non-integer $q$'s. Additionally, we also have to define
${\cal{W}}_\beta(M)$ for real-valued $\beta>1/2$, which we do in
Section~\ref{appendix:technical_results}, equation \eqref{eq:genSobolev}. Hence, throughout the paper
both $q$ and $\beta$ are understood as real-valued numbers. In particular,
all results and proofs hold also for $q,\beta\in\mathbb{N}$.

In practice, finding $(\hat\lambda,\hat q)$ that optimises the criteria in~\eqref{eq:criteria} consists of finding $\hat\lambda_q$ that solves, for each $q$ in some fine grid $\mathbb{Q}_n$, the criterion $T_\lambda(\lambda,q)$ up to an $o(1/n)$ factor, then finding $\hat q\in \mathbb{Q}_n$ that solves $T_q(\hat\lambda_q,q)$ and setting $(\hat\lambda,\hat q)=(\hat\lambda_{\hat q}, \hat q)$.
The grid $\mathbb{Q}_n=\{q_{1},\cdots,q_{N_n}\}\in(1/2,\log(n)]^{N_n}$ must be such that $|q_{i-1}-q_{i}|=o\{1/\log(n)\}$, $i=1,\dots,N_n$, with $q_{0}=1/2$.
This ensures that
\[
n^{-\frac{2q_{i-1}}{2q_{i-1}+1}} = n^{-\frac{2q_{i}}{2q_{i}+1}}\{1+o(1)\}, \quad i=1,\dots, N_n,
\]
which means that the discretisation is sufficiently fine.

\section[Asymptotics]{Asymptotics of the solutions of the estimating equations}\label{sec:asymptotics}%

Fix some continuous regression function $f\in L_2$, and denote $\bm{B}=(B_1,\ldots,B_n)^T=\bm{\Phi}^T\bm{f}$ such that $\mathbb{E}\bm{X}=\mathbb{E}\bm{\Phi}^T\bm{Y}=\bm{B}$.
The oracle smoothing parameters will be defined as a solution to the system of equations
\begin{align}
0&=\mathbb{E}T_\lambda(\lambda, q) =
\frac{1}{n}\bigg\{\sum_{i=q+1}^n\frac{B_i^2 \lambda n\eta_{q,i}}{(1+\lambda n\eta_{q,i})^2}-
\sum_{i=q+1}^n\frac{\sigma^2}{(1+\lambda n\eta_{q,i})^2}+o(1)\bigg\},\label{eq:EElambda}\\
0&=\mathbb{E}T_q(\lambda,q)=
\frac{1}{n}\sum_{i=q+1}^n\frac{B_i^2 \lambda n\eta_{q,i}\log(\lambda n\eta_{q,i})}{(1+\lambda n\eta_{q,i})^2} +
 \log(1/\lambda)\mathbb{E}T_\lambda(\lambda,q),\label{eq:EEq}
\end{align}
where the expectation is taken under model (\ref{eq:model}).
These expressions follow by several applications of
Lemmas~\ref{lemma:trace} and~\ref{lemma:trace_2}, similar to
derivations in \cite{Krivobokova:2013}.
To solve this system assume that for each $q>1/2$,
equation~\eqref{eq:EElambda} has a unique solution $\lambda_q$.
Then equation~\eqref{eq:EEq} at $\lambda=\lambda_q$ becomes
\begin{equation}\label{eq:EElambdaq}
0 = \mathbb{E}T_q(\lambda_q,q) =
\frac{1}{n}\sum_{i=q+1}^n\frac{B_i^2 \lambda_q n\eta_{q,i}\log(\lambda_q n\eta_{q,i})}{(1+\lambda_q n\eta_{q,i})^2}.
\end{equation}
If this equation has a unique solution ${\bar\beta}$, then the
solution to the system (\ref{eq:EElambda}), (\ref{eq:EEq}) on
$[1/n,1]\times(1/2,\log(n)]$ will be called the
oracle parameter $(\lambda_{\bar\beta},{\bar\beta})$.

Apparently, the risk \eqref{eq:risk} depends on the relationship between
$q$ and $\beta$, whereby $q$ should be chosen, while $\beta$ is
unknown. Therefore, we analyse both oracle parameters under two
scenarios: a {\it low order penalty} scenario where $q\le\max\{\beta>1/2:f\in\mathcal{W}_\beta(M)\}$, and
a {\it high order penalty} scenario where
$q>\max\{\beta>1/2:f\in\mathcal{W}_\beta(M)\}$. Here
$\max\{\beta>1/2:f\in\mathcal{W}_\beta(M)\}$ is considered, since 
 $f\in\mathcal{W}_\beta(M)$ does not preclude $f\in\mathcal{W}_{\beta'}(M)$ for some $\beta'>\beta$.
Additionally to $f\in\mathcal{W}_\beta(M)$, we also discuss the case
when $f$ is an analytic signal. $\mathcal{P}_\infty$ will denote the
space of all analytic functions on $[0,1]$ such that
$\mathcal{P}_\infty\subset\mathcal{W}_\infty(M)$, while the space of all polynomials of degree $d-1$ is denoted by $\mathcal{P}_d$, $d\in\mathbb{N}$.

\subsection[EBE for $\lambda$]{Empirical Bayes estimate for $\lambda$}\label{sec:asymptotics:empirical_bayes_lambda}

First we consider the solution to \eqref{eq:EElambda} for each $q>1/2$.
\begin{theorem}\label{th:lambda}
Let $f\in\mathcal{W}_\beta(M)$, and assume that $\|f^{(\beta)}\|^2>0$. \\
If
$1/2< q\leq \max\{\beta>1/2:f\in\mathcal{W}_\beta(M)\}$, then
\begin{equation}\label{eq:Bayes_oracle}
\lambda_q=\left[n\frac{\|f^{(q)}\|^2}{\sigma^2\kappa_q(0,2)}\{1+o(1)\}\right]^{-2q/(2q+1)},
\end{equation}
where the constants $\kappa_q(m,l)$ are defined in Section~\ref{appendix:technical_results}.\\
If $q>\max\{\beta>1/2:f\in\mathcal{W}_\beta(M)\}$, then 
\begin{equation}\label{eq:lambda_q_upper}
\lambda_q\ge\left[n\frac{\|f^{({\beta})}\|^2}{\sigma^2\kappa_q(0,2)}\{1+o(1)\}\right]^{-2q/(2{\beta}+1)}.
\end{equation}
Moreover, for any $q>1/2$, $\hat\lambda_q$ is consistent for
$\lambda_q$ and
\[
\lambda_q^{-1/(4q)}\big(\hat\lambda_q/\lambda_q - 1\big) \stackrel{d}{\longrightarrow}
\mathcal{N}\left[0,\; \frac{2 \kappa_q(2,2)}{\big\{3\kappa_q(0,2)-2\kappa_q(0,3)\big\}^2}\right],
\quad\text{ as } n\to\infty.
\]
\end{theorem}
Proof of equations \eqref{eq:Bayes_oracle} and
\eqref{eq:lambda_q_upper} follows from Lemmas \ref{lemma:trace}
and~\ref{lemma:quadratic_forms} in Section~\ref{appendix:technical_results}.
The consistency of $\hat\lambda_q$ and its asymptotic distribution in the case of $f\in\mathcal{W}_q(M)$ has been studied in \cite{Krivobokova:2013}.
Inspection of the  proofs in~\cite{Krivobokova:2013} shows that they hold with no changes for the case $q>\max\{\beta>1/2:f\in\mathcal{W}_\beta(M)\}$.\\
Note that if $f\in\mathcal{P}_{q}$ such that $\|f^{(q)}\|=0$, then $\lambda_q=\infty$.

\subsection[EBE for $q$]{Empirical Bayes estimate for $q$}\label{sec:asymptotics:empirical_bayes_q}

First, consider the low penalty scenario where
$1/2<q\le\max\{\beta>1/2:f\in\mathcal{W}_\beta(M)\}$ holds, so that in particular $f\in\mathcal{W}_q(M)$.
By Lemma~\ref{lemma:quadratic_forms} (cf.~Section~\ref{appendix:technical_results}),
\begin{equation}\label{eq:EQforBeta}
\mathbb{E}T_q(\lambda_q,q)=-\lambda_q\log(1/\lambda_q)\|f^{(q)}\|^2\{1+o(1)\}.
\end{equation}
Hence, for all $1/2<q\le\max\{\beta>1/2:f\in\mathcal{W}_\beta(M)\}$ the estimating
equation 
$\mathbb{E}T_q(\lambda_q,q)$ remains strictly negative for
$f\in\mathcal{W}_q(M)\backslash\mathcal{P}_q$ (that is as long as
$\|f^{(q)}\|\neq 0$). If $f\in\mathcal{P}_q$,
then $\mathbb{E}T_q(\lambda_q,q)=0$.

Consider now the high order penalty scenario where $q>\max\{\beta>1/2:f\in\mathcal{W}_\beta(M)\}$.
Contrary to the low penalty scenario, the sign of~\eqref{eq:EElambdaq}
is not characterised by just the assumption
$f\in\mathcal{W}_\beta(M)$, which implies $f\not\in\mathcal{W}_q(M)$.
It turns out, that not every signal $f$ that belongs to $\mathcal{W}_\beta(M)$ but not to $\mathcal{W}_{\beta+\delta}(M)$ for any $\delta>0$ will be such that~\eqref{eq:EElambdaq} is positive for $q>\max\{\beta>1/2:f\in\mathcal{W}_\beta(M)\}$.
Such a mismatch between smoothness as ``measured'' by $\max\{\beta>1/2:f\in\mathcal{W}_\beta(M)\}$, and smoothness as ``measured'' by a change in the sign of the sum in~\eqref{eq:EElambdaq} (which we can estimate), seems unavoidable (cf.~\citealp{Belitser:2008} for a similar issue in the context of hypothesis testing for smoothness in the Gaussian white noise model, and~\citealp{Gine:2010} in the context of H\"older smoothness in the construction of adaptive $L_\infty$ credible bands in density estimation).
From such issues stems, for example, the inability to construct adaptive credible sets in certain models with good coverage probability for $f\in{\bigcup}_{\beta\in{B}}\mathcal{W}_\beta$ if the range of smoothnesses ${B}$ is large (cf.~\citealp{Low:1997} and Section~\ref{sec:coverage}).

To ``estimate'' $\max\{\beta>1/2:f\in\mathcal{W}_\beta(M)\}$ for as large a family of models as possible, it is customary to remove from each model the functions for which the mismatch occurs, and thus consider the estimation problem over a smaller family.
Possible functions to remove are those which are not \emph{self-similar} (cf.~\citealp{Picard:2000,Gine:2010}), or that do not satisfy a \emph{polished-tail} condition (cf.~\citealp{Szabo:2013}).

In our context, the set of functions with polished tails, call it $\mathcal{M}=\mathcal{M}(L,N,\rho)$, $L>0$, $N\in\mathbb{N}$, $\rho\ge2$, corresponds the set of all square integrable sequences such that
\begin{equation}\label{eq:self_similar}
\frac1n\sum_{i=j}^n B_i^2 \le \frac Ln \sum_{i=j}^{\rho j} B_i^2, \quad N\le j \le n/\rho.
\end{equation}
Such a condition has the role of excluding ``irregular'' signals, and unlike self-similarity conditions, it is not associated with any specific smoothness class.
The condition ensures that the energy contained in the blocks $(B_j,\dots,B_{\rho j})$ does not surge over contributions of earlier blocks -- the signal can only get more ``polished'' as one runs along the sequence $B_i$.
This effectively excludes irregular signals that contain artefacts in their tails that influence the smoothness of the signal but are not detectable due to the noise.

For signals with polished tails, the criterium $T_q$ can actually pick up on the smoothness of the signal.
For fixed parameters $L,N,\rho$, there is a well defined notion of smoothness which we denote
\begin{equation}\label{eq:smoothness}
\bar\beta = \max\big\{\beta>1/2:f\in\mathcal{W}_\beta(M)\cap\mathcal{M}\big\}, \quad\quad
\end{equation}
If $f$ satisfies the polished tail conditions and is not in $\mathcal{W}_{\beta}(M)$, then by Lemma~\ref{lemma:quadratic_forms}, for some $c>0$,
\begin{equation}\label{eq:EQforBeta3}
\mathbb{E}T_q(\lambda_q,q) \ge c \lambda_q^{\beta/q} > 0,
\end{equation}
for all large enough $n$.
The conclusion is that for all sufficiently large $n$, $\mathbb{E}T_q(\lambda_q,q)$ changes signs at $\bar\beta$.\\

We conclude that the behaviour of $\mathbb{E}T_q(\lambda_q,q)$ can be described as follows.
If for some $\beta>1/2$, $f\in\mathcal{W}_\beta(M)\cap\mathcal{M}$ 
then the (continuous) criterion $\mathbb{E}T_q(\lambda_q,q)$ has a zero at $\min\{q\in \mathbb{Q}_n: q>\bar\beta\}$ 
since it is negative for $q\le\bar\beta$, and positive for $q>\bar\beta$.
If $f\in\mathcal{P}_\infty$ then $\mathbb{E}T_q(\lambda_q,q)\le 0$, $q\in \mathbb{Q}_n$;
if for some $d\in\mathbb{N}$,
$f\in\mathcal{P}_d\backslash\mathcal{P}_{d-1}$ (such that $f$ is a
polynomial of degree exactly $d-1$), then $\mathbb{E}T_q(\lambda_q,q)$
is negative for $1/2<q\leq d-1$ and is zero for $q\geq d$.

The proof of the following theorem is in Section~\ref{appendix:proof_consistency_q}.
\begin{theorem}\label{theorem:consistency_q}
Assume that for some $\beta>1/2$, $f\in\mathcal{W}_\beta(M)$, then
\[
\mathbb{P}\big\{ \beta< \hat q \le \log(n) \big\} \to 1, \quad n\to\infty.
\]
If furthermore $\beta=\bar\beta$ as defined
in~\eqref{eq:smoothness}, and $f\in\mathcal{M}$, then
\[
\hat q \stackrel{P}{\longrightarrow} \bar\beta, \quad n\to\infty,
\]
If for some $d\in\mathbb{N}$, $f\in\mathcal{P}_d$ then
\[
\mathbb{P}\big\{ d \le \hat q \le \log(n) \big\} \to 1, \quad n\to\infty.
\]
\end{theorem}

As a side-note, the oracle $\lambda_q$ can also be made more explicit for $f\in\mathcal{W}_{\beta}(M)\cap\mathcal{M}$.
By~\eqref{eq:lambda_q_upper} and \eqref{eq:self_similar}, one finds that for each $f\in\mathcal{M}_{\bar\beta}(M)\cap\mathcal{M}$,
\begin{equation}\label{eq:lambda_q_self_similar}
\lambda_q(f) = \left[n\frac{c_f M^2}{\sigma^2\kappa_q(0,2)}\{1+o(1)\}\right]^{-2q/(2\bar\beta+1)},
\end{equation}
for some constant $0<c_f\le1$, which depends on $f$ and is bounded away from zero uniformly over $f\in\mathcal{W}_\beta(M)\cap\mathcal{M}$.

The exclusion of certain signals (so that the behaviour above holds) can be argued to be innocuous since the set of all signals that do not satisfy the polished-tail condition (or that are not self-similar) is ``small''.
This can be justified following several arguments:
removing such signals leaves the minimax rate (almost) unchanged so that the statistical problem does not become simpler;
the probability that a function sampled from the prior does not satisfy such conditions is zero;
there are also topological arguments for this.
For a more extensive discussion cf.~\citealp{Gine:2010,Szabo:2013} and the references therein.
However, from the practical perspective, exactly which signals are removed is not relevant since one cannot check if the data come from a regression function  satisfying a polished tail condition or not.
In this sense one might as well implicitly exclude all signals for which~\eqref{eq:EQforBeta3} or~\eqref{eq:lambda_q_self_similar} do not hold.\\

\section[Confidence sets]{Bayesian credible sets as adaptive confidence sets}\label{sec:coverage}%

In Section~\ref{sec:empirical_bayes} we propose a method for selecting the penalty order $q$ of smoothing splines and the corresponding smoothing parameter $\lambda_q$.
An immediate application of the consistency results in Section~\ref{sec:asymptotics} is that $\hat q$ and $\hat\lambda_{\hat q}$ can be directly plugged into the smoothing spline $\hat{f}_{\lambda,q}$ to obtain adaptive estimates for any continuous regression function $f\in L_2$.
(This follows immediately from the consistency of the parameters, and standard arguments for smoothing spline estimators; cf.~\citealp{Wahba:1990}.)
In this section we present another application: the construction of rate adaptive confidence sets based on the empirical Bayes posterior~\eqref{eq:empirical_Bayes_posterior}.

One of the often mentioned advantages of the Bayesian approach is that a posterior distribution provides statisticians with more than just point estimates.
For appropriate priors, if the data are distributed according to a fixed distribution in the model, then with probability going to one, the posterior concentrates around this distribution.
If this is the case, then for appropriate $q$ and $\lambda$ 
a small $\ell^2$-ball centred at the posterior mean $\bm{\hat{f}}_{\lambda,q}$, can capture most of the mass of the marginal posterior for $\bm{f}$.
Since for each $\lambda$ and $q$ the posterior is known explicitly, simulating a high-probability region of the posterior (a credible set) and using it as a \emph{frequentist} confidence set is of great appeal.
However, it is known that such credible sets do not always have good frequentist coverage properties.
This is more so the case when dealing with posteriors that adapt to the smoothness of the underlying signal to be estimated.
In this section we adapt a technique developed by~\cite{Szabo:2013} for the Gaussian white noise model, to study the behaviour of a specific Bayesian credible set for our regression model~\eqref{eq:model}.
Complicating factors in our setup are that the variance of the noise is not assumed to be known, and that we work with two empirically chosen parameters ($\lambda$ and $q$) simultaneously.
We outline the technique in some detail since it is of independent interest.

We remind that the marginal posterior $\Pi_{\lambda,q}^{\bm{f}}\big(\,\cdot\big|\bm{Y},\bm{x}\big)$ equals
\[
t_n\Big(
\bm{\hat{f}}_{\lambda,q},\, \hat\sigma_{\lambda,q}^2 \bm{S}_{\lambda,q}\Big), \quad\text{where}\quad
\hat\sigma_{\lambda,q}^2 = 
\frac1n\bm{Y}^T\big(\bm{I}-\bm{S}_{\lambda,q}\big)\bm{Y}.
\]
Representation properties of the multivariate $t$-distribution state
that if $\bm{f}$ is distributed like the marginal posterior above,
then
$\|\bm{f}-\bm{\hat{f}}_{\lambda,q}\|^2/\hat\sigma_{\lambda,q}^2$
 is distributed like $\bm{Z}^T\bm{S}_{\lambda,q}\bm{Z}/N$, where $\bm{Z}\sim N(\bm{0},\bm{I}_n)$, $N\sim \mathcal{X}_n^2$, and $N$ is independent of $\bm{Z}$.
Conclude that for any $\alpha\in(0,1)$ there exists a (known, deterministic) sequence $r_n(\lambda,q)$ such that for every $n,\lambda,q$,
\[
\Pi_{\lambda,q}^{\bm{f}}\Big( \|\bm{f}-\bm{\hat{f}}_{\lambda,q}\| \le \hat{\sigma}_{\lambda,q}\, r_n(\lambda,q) \Big|\bm{Y},\bm{x}\Big) = 1-\alpha.
\]
The level $\alpha$ is fixed for the remainder of this section.
It is therefore natural to consider, for any $L\ge1$, the empirical credible ball
\begin{equation}\label{eq:credible_set}
\hat{\mathcal{C}}_n(L) = \Big\{ \bm{f}\in\ell^2: \|\bm{f}-\bm{\hat{f}}_{\hat\lambda_{\hat q},\hat q}\| \le \hat{\sigma}_{\hat\lambda_{\hat q},\hat q}\, L\, r_n(\hat\lambda_{\hat q}, \hat q)  \Big\}.
\end{equation}
By definition of the sequence $r_n(\lambda,q)$, for any $L\ge1$,
\[
\Pi_{\hat\lambda_{\hat q},\hat q}^{\bm{f}}\big(\, \hat{\mathcal{C}}_n(L)\, \big|\bm{Y},\bm{x}\big) \ge 1-\alpha,
\]
such that we can sample functions in $\hat{\mathcal{C}}_n(L)$ by sampling functions from the posterior and then keeping those that satisfy the inequality in~\eqref{eq:credible_set}.
Such functions give a visual impression of the uncertainty in the
point estimate $\bm{\hat{f}}_{\hat\lambda_{\hat q}, \hat q}$ -- the
adaptive empirical Bayesian smoothing spline.
Note that since for each $\lambda$ and $q$ the posterior is known explicitly, simulating $\hat{\mathcal{C}}_n(L)$ is straightforward.

The following theorem is proved in Section~\ref{appendix:coverage} of Section~\ref{appendix:technical_results}.
\begin{theorem}\label{theorem:coverage}
Consider an interval $B=[\underline b,\bar b]$, where $1/2<\underline b \le \bar b < \infty$, and define $\mathcal{W}=\bigcup_{\beta\in B}\mathcal{W}_\beta(M)\cap\mathcal{M}(L,N,\rho)$.
Then, for all large enough $L$,
\begin{align}
\inf_{\bm{f}\in\mathcal{W}}
\mathbb{P}_f \big\{ \bm{f} \in \hat{\mathcal{C}}_n(L) \big\}				&= 1+ o(1),
\quad\text{and}\quad	\label{eq:honest_coverage}\\
\inf_{\bm{f}\in\mathcal{W}_\beta(M)}
\mathbb{P}_f \big\{r_n(\hat\lambda_{\hat q},\hat q) \le K n^{-\beta/(2\beta+1)} \big\}	&=1+o(1),	\quad \beta\in B,\label{eq:optimal_radius}
\end{align}
for some large enough constant $K>0$, depending on $L$, $M$, $\rho$, $\sigma^2$, $\underline b$, and $\bar b$.
\end{theorem}
Statement~\eqref{eq:honest_coverage} is usually referred to as~\emph{honest coverage},
while~\eqref{eq:optimal_radius} means that the credible ball
$\hat{\mathcal{C}}_n(L)$ has a radius of the optimal order.

Ideally one would like to take $\mathcal{W}=\ell^2$.
However, as mentioned in Section~\ref{sec:asymptotics:empirical_bayes_q}, it is known (cf.~\citealp{Low:1997}) that it is in general not possible to fulfil conditions~\eqref{eq:honest_coverage} and~\eqref{eq:optimal_radius} simultaneously if $\bar b/\underline b>2$ and $\mathcal{W}=\ell^2$.
To allow~\eqref{eq:honest_coverage} and~\eqref{eq:optimal_radius} to hold simultaneously for a wide (but bounded) range of smoothnesses one usually identifies 
``problematic'' functions which are either removed from the model 
(cf.~\citealp{Gine:2010,Szabo:2013}) or replaced with a collection 
of so called surrogates, ``non-problematic'' replacement functions that retain the main features of interest of the functions that were removed from the model 
(cf.~\citealp{Genovese:2008}).
Imposing an upper bound $\bar b$ on the smoothness is also necessary if we are to establish~\eqref{eq:honest_coverage} and~\eqref{eq:optimal_radius}.
Such a bound can also be justified from a computational standpoint.

The constant $L$, which is the multiplicative factor for the radius of the credible set $\hat{\mathcal{C}}_n$, must be taken appropriately large for~\eqref{eq:honest_coverage} and~\eqref{eq:optimal_radius} to hold.
It is possible to provide an explicit lower-bound $L$ by inspecting the constants in the proof of Theorem~\ref{theorem:coverage}: 
for all sufficiently large $n$ we may take
\[
L \ge
1 + \big\{\kappa_q(0,2)/\kappa_q(0,1)\big\}^{1/2} = 1 + \big\{(2q-1)/(2q)\big\}^{1/2},
\]
so that uniformly over $q$, $L\ge2$. 
Inspection of the computations in Section~\ref{appendix:technical_results} shows that the level $\alpha$ appears associated with lower order terms so that $L$ does not depend on $\alpha$, even if we were to allow $L$ to depend on $q$.
Because of this one does not get exactly coverage $1-\alpha$ and the credible sets $\hat{\mathcal{C}}_n$ are always conservative (in that the asymptotic probability of coverage is $1>1-\alpha$).\\

It follows, in fact, from the inequalities established in the proof of Theorem 3 that the posterior distribution contracts (with respect to $\|\cdot\|$) around the true regression function at the optimal rate for $f\in\mathcal{W}_\beta$, when $f$ is indeed in this space.
To establish this, condition~\eqref{eq:self_similar} is not needed.

\section[Comparison with GCV]{Comparison with frequentist smoothing splines}\label{sec:GCV}%

In the frequentist framework there are several competing ways of selecting the smoothing parameter $\lambda$ (for a fixed $q$).
Typically, $\lambda$ is selected as a minimiser of some asymptotically unbiased estimator of the risk $\mathbb{E}\|\bm{\hat f}_{\lambda,q}-\bm{f}\|^2$.
Generalised cross validation, Mallow's $C_p$ and Akaike's information
criterion are particular examples of such estimates; let
$\hat\lambda_f$ denote a minimiser of one of such criteria. 
If the regression function $f$ belongs to $\mathcal{W}_{\beta}(M)$, and we set $q\in[\beta/2,\beta]$ such that $\beta/q\in[1,2]$, then $\hat\lambda_f$ adapts.
This means that generalised cross-validated smoothing splines adapt to the unknown smoothness $\beta$ in the sense that the estimator $\hat\lambda_f$ is consistent for the oracle
\begin{equation}\label{eq:GCV_oracle}
\lambda_f \ge \left[n\frac{\|f^{(\beta)}\|^2}{\sigma^2\kappa_{q}(1,2)}\{1+o(1)\}\right]^{-2q/(2\beta+1)}.
\end{equation}
Furthermore, Theorem 3 of~\cite{Krivobokova:2013} states that
\[
\lambda_f^{-1/(4q)}\big(\hat\lambda_f/\lambda_f - 1\big) \stackrel{d}{\longrightarrow}
\mathcal{N}\bigg[0,\; \frac{2\kappa_q(4,2)}{\big\{4\kappa_q(1,2)-3\kappa_q(1,3)\big\}^2}\bigg],
\quad\text{ as } n\to\infty.
\]
so that the asymptotic variance above can be much larger than that
associated with the empirical Bayes estimate $\hat\lambda_{\beta}$ for
the range of values of $q$ for which the GCV $\hat\lambda_f$ adapts,
see \cite{Krivobokova:2013} for more discussion and simulations.

In this section we investigate how the asymptotic variance of $\hat\lambda_f$ compares to that of $\hat\lambda_{\hat q}$;
or, more specifically, we compare the distances of $\hat{f}_{\hat\lambda_f,q}$ and of $\hat{f}_{\hat\lambda_{\hat q}, \hat q}$ to the true regression funtion.
We use the credible sets from the previous section as a proxy for this comparison:
we bound the probability that the regression function belongs to a ball centered at $\hat{f}_{\hat\lambda_f,q}$ with a radius that assures coverage for the Bayesian credible set $\hat{\mathcal{C}}_n(L)$ -- by construction, that is the radius of $\hat{\mathcal{C}}_n(1)$, since we are only interested here in coverage, and not honest coverage.

The proof of the following theorem can be found in Section~\ref{appendix:GCV}.
\begin{theorem}\label{theorem:GCV}
Assume that the regression function $f$ belongs to $\mathcal{W}_\beta(M)$, $\beta>1/2$, such that with probability going to 1 the radius of the credible set $\hat{\mathcal{C}}_n(1)$ is $\sigma\, r_n(\lambda_\beta,\beta)$.
Define
\[
\hat{\mathcal{D}}_n = \hat{\mathcal{D}}_n\big(\bm{\hat f}_{\hat\lambda_f,q}\big) =
\Big\{ \bm{f}\in\ell^2: \|\bm{f}-\bm{\hat f}_{\hat\lambda_f,q}\|\le \sigma\, r_n(\lambda_\beta, \beta) \Big\}.
\]
Then, for any $q$ such that $\hat\lambda_f$ adapts to the smoothness $\beta$, i.e., for any $q\in[\beta/2,\beta]$,
\[
\mathbb{P}_f\Big\{ \bm{f}\in \hat{\mathcal{D}}_n\big(\bm{\hat f}_{\hat\lambda_f,q}\big) \Big\} = o(1), \quad n\to\infty.
\]
\end{theorem}

The conclusion is that $\hat\lambda_f$ can somewhat adapt to the smoothness of the regression function, but at the cost of a high asymptotic variance.
Using the fact that for each fixed $q$, the empirical Bayes selected $\hat\lambda_q$ has much lower asymptotic variance, we show that the smoothing parameter $\hat\lambda_{\hat q}$ outperforms $\hat\lambda_f$:
if the centre of the empirical Bayes credible ball (the adaptive
empirical Bayesian smoothing spline) is replaced by a frequentist smoothing
spline with the smoothing parameter $\hat\lambda_f$, then the coverage property is lost.
Note that even if this were not the case, the empirical Bayes smoothing spline would still adapt to a wider range of smoothnesses than the risk-based smoothing spline (which adapts only within a $[q,2q]$ range).

\section[Numerical simulations]{Numerical
  simulations}\label{sec:simulations}
The following simulation study aims to verify our theoretical findings
in finite samples and compare frequentist and
proposed adaptive empirical Bayesian smoothing splines in terms of the
average mean squared error.

In all settings the Monte-Carlo sample $M=1000$, the sample size is $n=1000$, the design points are
fixed and equidistant $x=i/n$, $i=1,\ldots,n$ and $\sigma=0.1^2$. The results for other sample
sizes and higher $\sigma$s were found conceptually similar and are not
reported. We consider two mean functions $f_j$, $j=1,2$ that are
scaled by its range. Function $f_1$, shown on the top left plot of
Figure \ref{figure:f1}, has a known decay of its
Demmler-Reinsch coefficients, namely we set
\[
f_1(x)=\sum_{i=q+1}^n\psi_{\beta,i}(x)(i+1)^{-\beta}\cos(2i),\;\;\beta=3,
\]
where $\psi_{\beta,i}$ is the $i$th basis function of the Demmler-Reinsch
basis of degree $\beta=3$, defined in
(\ref{eq:def_demmler_basis_functions1}). The second function is
analytical $f_2(x)=\cos(5\pi x)$ (Figure \ref{figure:f2}) with an exponential decay of its
Demmler-Reinsch coefficients.

We estimated both functions by smoothing splines with the empirical
Bayesian smoothing parameter $\hat\lambda_{\hat{q}}$ and with the
smoothing parameter $\hat\lambda_f$.
To get $\hat\lambda_{\hat{q}}$, first $\hat\lambda_q$ is obtained for
$q=1,2,\ldots,6$ as a solution to $T_\lambda(\lambda, q)=0$ and then
$\hat{q}$ is set to the nearest integer
to $q^*$ such that
$T_q(\hat\lambda_q,q^*)=0$. 
\begin{table}[b!]
\begin{tabular}{l|cccccc}
&&GCV&&&&\\
&EB& $q=2$& $q=3$& $q=4$& $q=5$&
                                                                  $q=6$\\\hline
 $f_1$& &&&&&\\
$\mbox{mean}(\hat\lambda)$&$5.7\cdot10^{-12}$& $1.8\cdot10^{-08}$&
                                                                  $7.1\cdot10^{-12}$& $1.8\cdot10^{-15}$& $4.5\cdot10^{-19}$& $1.2\cdot10^{-22}$\\
$\mbox{var}(\hat\lambda)$&$5.0\cdot10^{-25}$& $2.3\cdot10^{-17}$&
                                                                   $1.2\cdot10^{-23}$&
                                                                                       $1.7\cdot10^{-30}$&
                                                                                                           $2.0\cdot10^{-37}$&$ 2.4\cdot10^{-44}$ \\\hline
$A(\hat\lambda)$&$2.5\cdot10^{-06}$&$ 2.9\cdot10^{-06}$& $2.6\cdot10^{-06}$& $2.7\cdot10^{-06}$&$ 2.9\cdot10^{-06}$&$ 2.9\cdot10^{-06}$\\ 
$R$&
                                                                   $-$&$1.147912$
       &$1.032582$& $1.075522$ &$1.119951$& $1.139055$ 
\\\hline
$f_2$& &&&&&\\
$\mbox{mean}(\hat\lambda)$&$1.7\cdot10^{-19}$& $5.9\cdot10^{-09}$&
                                                                  $1.6\cdot10^{-11}$& $3.4\cdot10^{-14}$& $3.6\cdot10^{-17}$& $2.8\cdot10^{-19}$\\
$\mbox{var}(\hat\lambda)$&$1.1\cdot10^{-36}$& $2.0\cdot10^{-18}$&
                                                                   $2.4\cdot10^{-23}$&
                                                                                       $4.8\cdot10^{-28}$&
                                                                                                           $4.0\cdot10^{-34}$&$ 2.4\cdot10^{-38}$ \\\hline
$A(\hat\lambda)$&$1.5\cdot10^{-06}$& $3.8\cdot10^{-06}$& $2.1\cdot10^{-06}$& $1.9\cdot10^{-06}$& $1.8\cdot10^{-06}$& $1.6\cdot10^{-06}$\\ 
$R$&
                                                                   $-$  &$2.483676$& $1.406107$ &$1.266803$& $1.164454$ &$1.040263$
\end{tabular}
\caption{Simulation results for functions $f_1$ and $f_2$.}
\label{table:sim}
\end{table}
The smoothing parameter
$\hat\lambda_f$ has been calculated by
generalized cross-validation for different values of $q=2,\ldots,6$.

We compare the resulting smoothing parameter estimators
$\hat\lambda_{\hat{q}}$ and $\hat\lambda_f$ for $q=2,\ldots,6$ in
terms of the sample mean and sample variance in the Monte Carlo sample
(see 
$\mbox{mean}(\hat\lambda)$ and $\mbox{var}(\hat\lambda)$ in Table
\ref{table:sim}). 
Further, we compare the empirical average mean
squared error (AMSE)
  of the resulting estimators, that is 
\[
A(\hat\lambda)=\frac{1}{Mn}\sum_{i=1}^M\sum_{j=1}^n\left\{\hat{f}_i(x_j,\hat\lambda)-f(x_j)\right\}^2,
\]
where $\hat{f}_i$ denotes a smoothing spline estimator in the $i$th Monte
Carlo simulation. The values of $A(\hat\lambda)$ are
given in Table \ref{table:sim} together with the AMSE ratio,
which is defined as $R=A(\hat\lambda_f)/A(\hat\lambda_{\hat{q}})$, so
that values of $R>1$ imply superiority of the adaptive empirical
Bayesian estimator.

Let us first consider the simulation results for $f_1$. 
According to the theoretical
results on $\hat\lambda_f$ discussed in Section \ref{sec:GCV}, as long
as $q\geq\beta/2$, its oracle $\lambda_f\geq
\mbox{const}\;n^{-2q/(2\beta+1)}$, leading to $\mathbb{E}\|\bm{\hat
  f}(\hat\lambda_f) - \bm{f}\|^2 \asymp
n^{-2\beta/(2\beta+1)}$. For $\beta=3$ we expect
to observe that $\lambda_f$ for $q=2$ is larger than for $q=3$ and
goes very fast to zero for $q>\beta=3$. Moreover, $\hat{f}(\hat\lambda)$
for $\hat\lambda_{\hat{q}}$ and $\hat\lambda_f$ for $q=2,\ldots,6$ should
all have the same convergence rates and differ only in constants. Results
of the simulations given in Table \ref{table:sim} confirm these
theoretical findings. We observe also that the means of $\hat\lambda_{\hat{q}}$ and $\hat\lambda_f$ for
$q=3$ are of the same rate, but the variance of $\hat\lambda_f$ is
much higher, which is also visibile in the boxplots, given on the
bottom left plot of Figure \ref{figure:f1}. The differences between
$A(\hat\lambda)$, shown on the bottom right plot of Figure
\ref{figure:f1}, are less pronounced, but from the AMSE ratio given in
Table \ref{table:sim} we find that the empirical Bayesian smoothing spline estimator
outperforms the frequentist smoothing spline uniformly in
$q$. The smallest difference is observed between
$A(\hat\lambda_{\hat{q}})$ and $A(\hat\lambda_f)$ for the true
$q=\beta=3$, which is, of course, unknown in practice and is not
estimated in the frequentist framework. 
\begin{figure}[t!]
\begin{tabular}{cc}
\vspace{0.3cm}\hspace{-6cm}(a) &\hspace{-3.5cm}(b) \\
\includegraphics[width=0.45\textwidth]{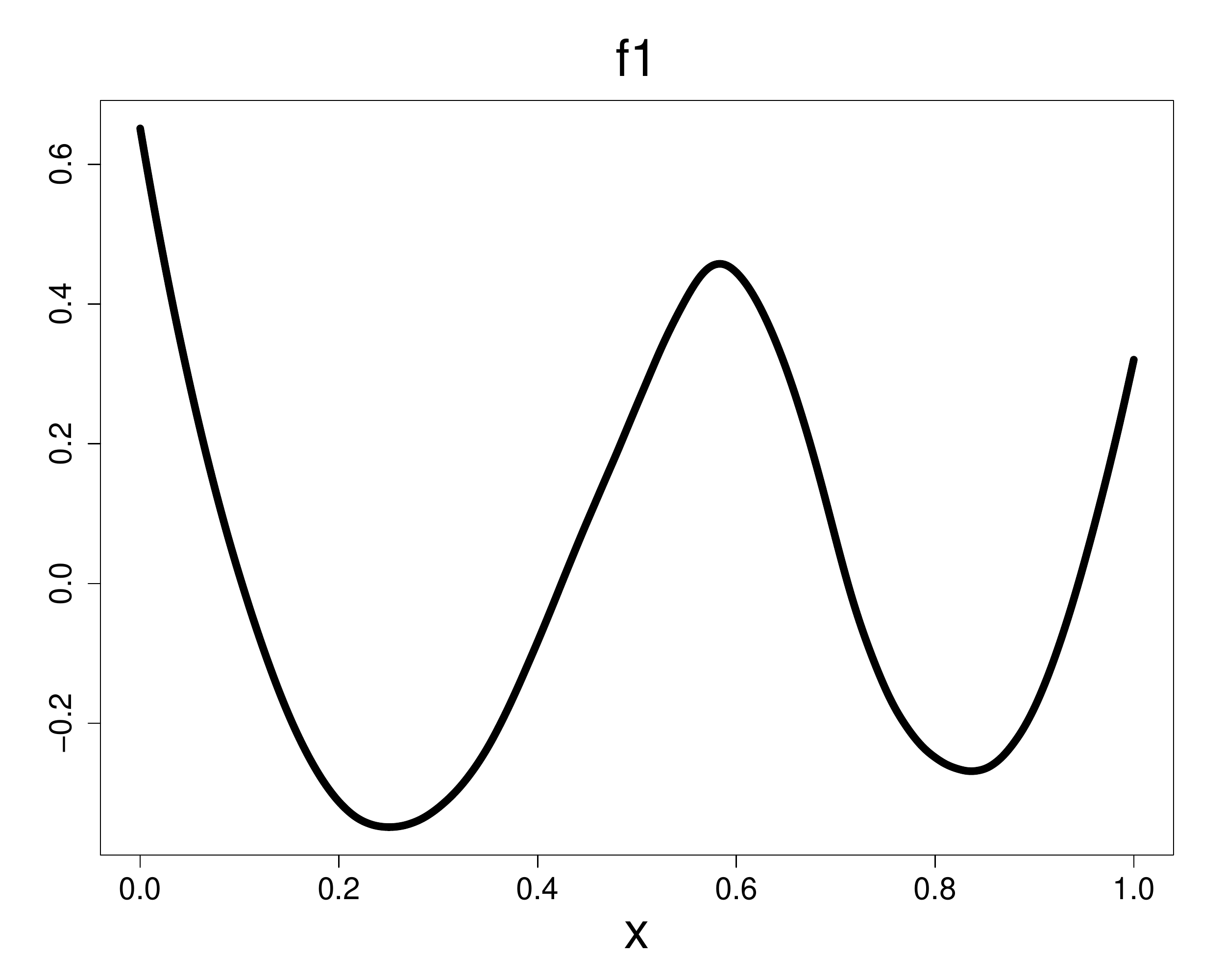}
\includegraphics[width=0.45\textwidth]{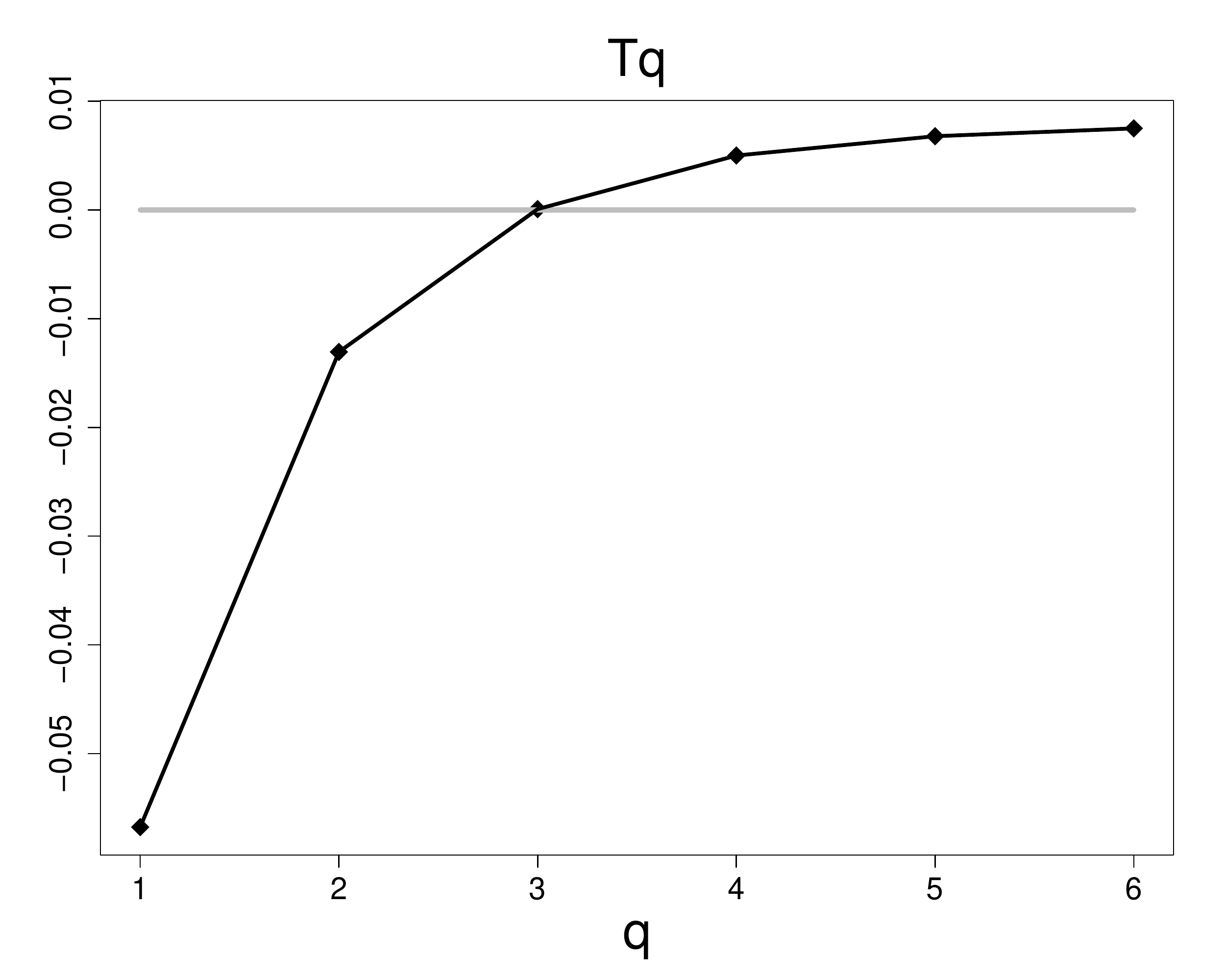}\\
\vspace{0.3cm}\hspace{-6cm}(c) &\hspace{-3.5cm}(d) \\
\includegraphics[width=0.45\textwidth]{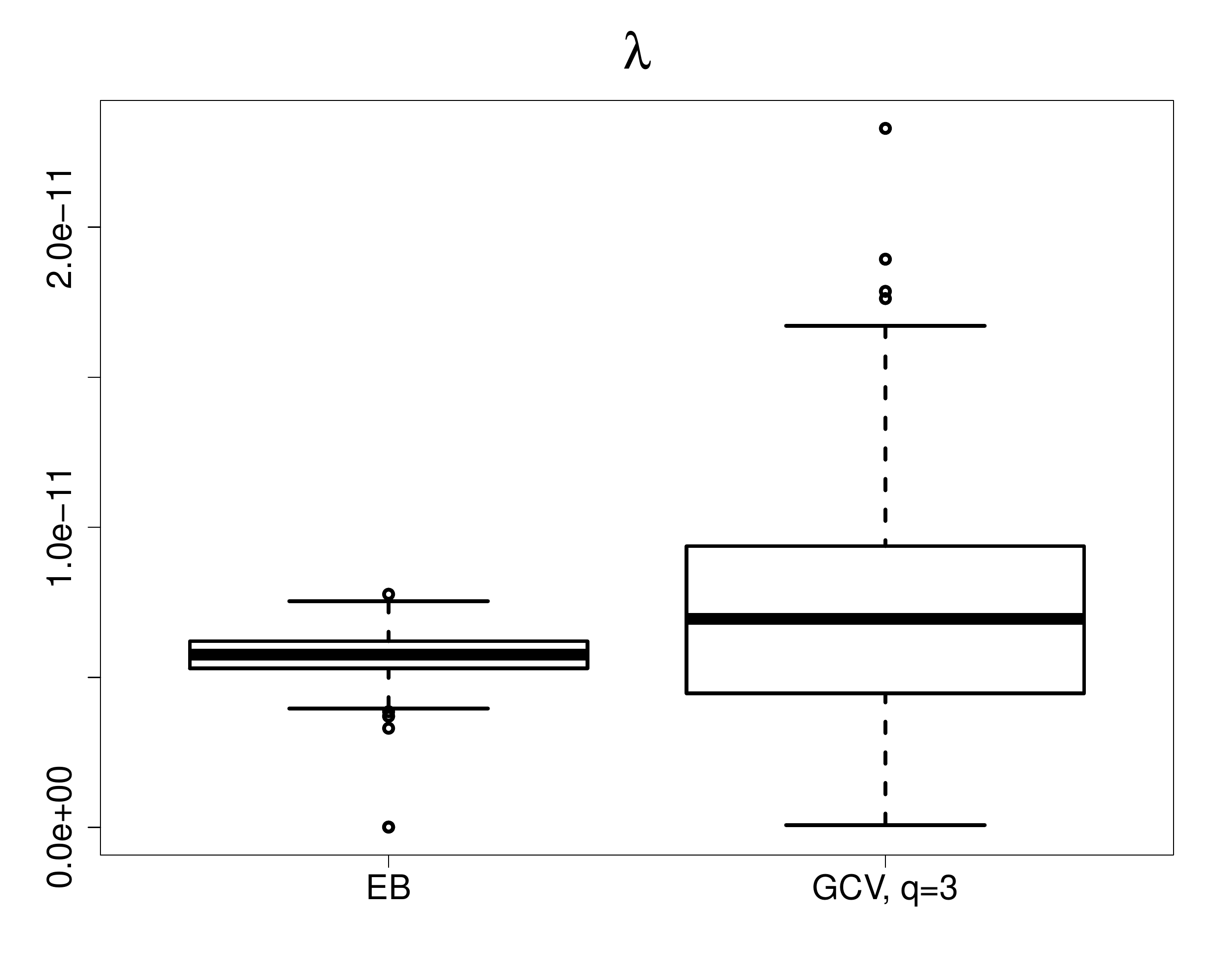}
\includegraphics[width=0.45\textwidth]{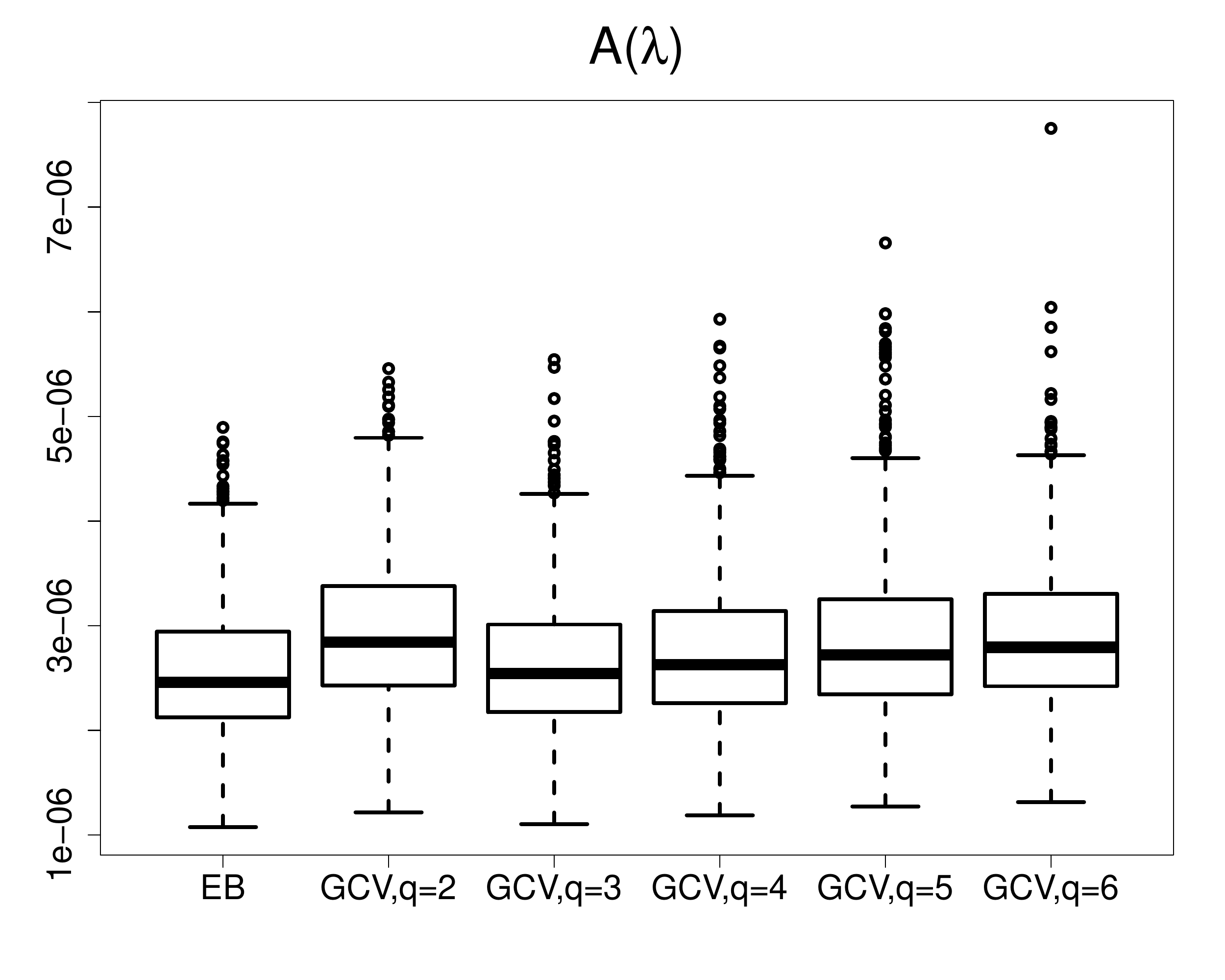}
\end{tabular}
\caption{(a) Function $f_1$, (b) Criterion $T_q(\hat\lambda_q,q)$ for
  $f_1$, (c) Boxplots of $\hat\lambda_{\hat{q}}$ and $\hat\lambda_f$ obtained with GCV and
  $q=\beta=3$, (d) Boxplots of $A(\hat\lambda_{\hat{q}})$ (EB) and
  $A(\hat\lambda_f)$}
\label{figure:f1}
\end{figure}
Finally, we
remark that the empirical
Bayesian estimator of $q$ appeared to be very reliable for this value
of $\beta$: out of $M=1000$ samples in $998$ cases $\hat{q}=3$ has
been obtained and in two cases $\hat{q}=4$. The estimating equation
$T_q(\hat\lambda_q,q)$ is shown on the top right plot of
Figure \ref{figure:f1}. 

\begin{figure}[t!]
\begin{tabular}{cc}
\vspace{0.3cm}\hspace{-6cm}(a) &\hspace{-3.5cm}(b) \\
\includegraphics[width=0.45\textwidth]{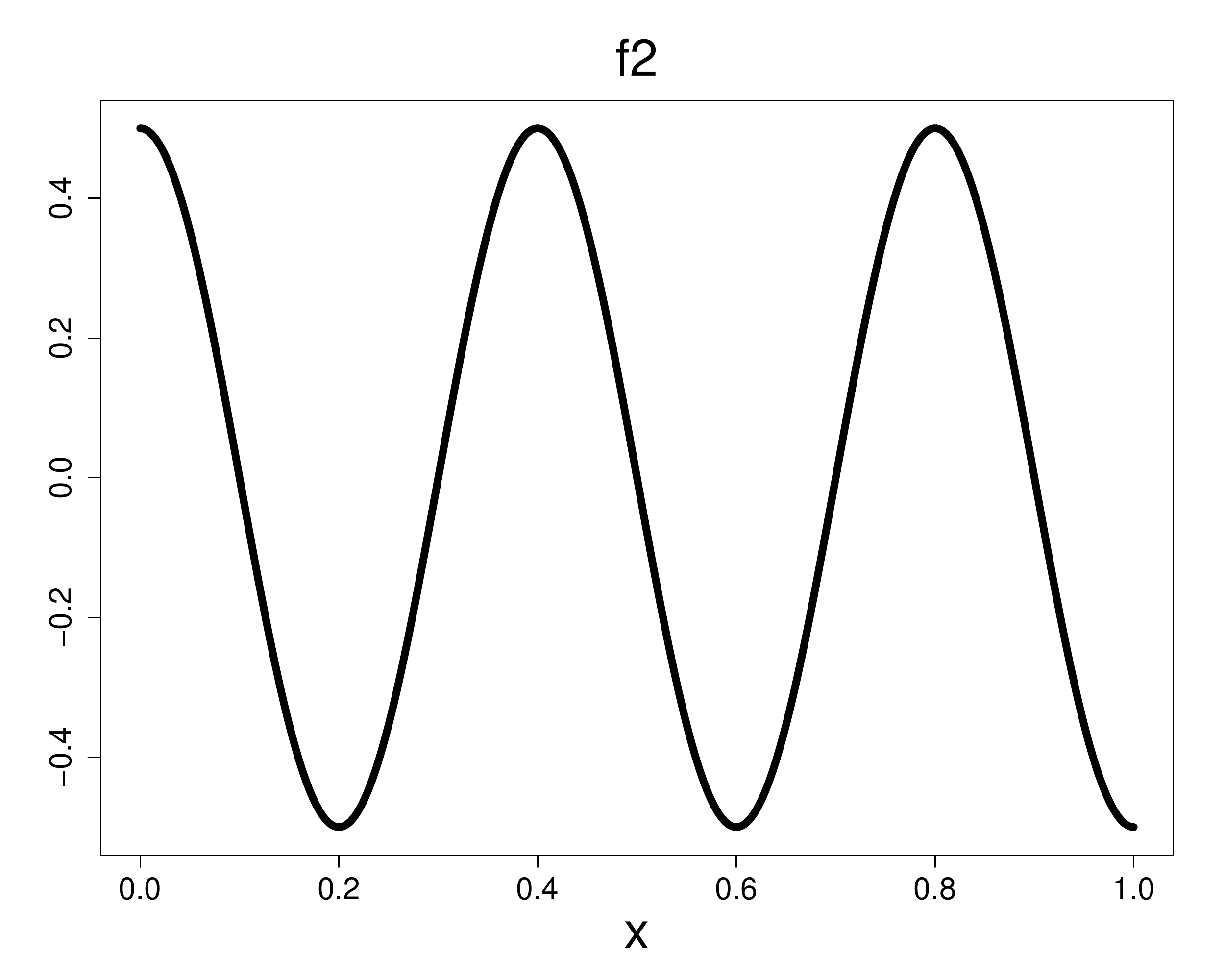}
\includegraphics[width=0.45\textwidth]{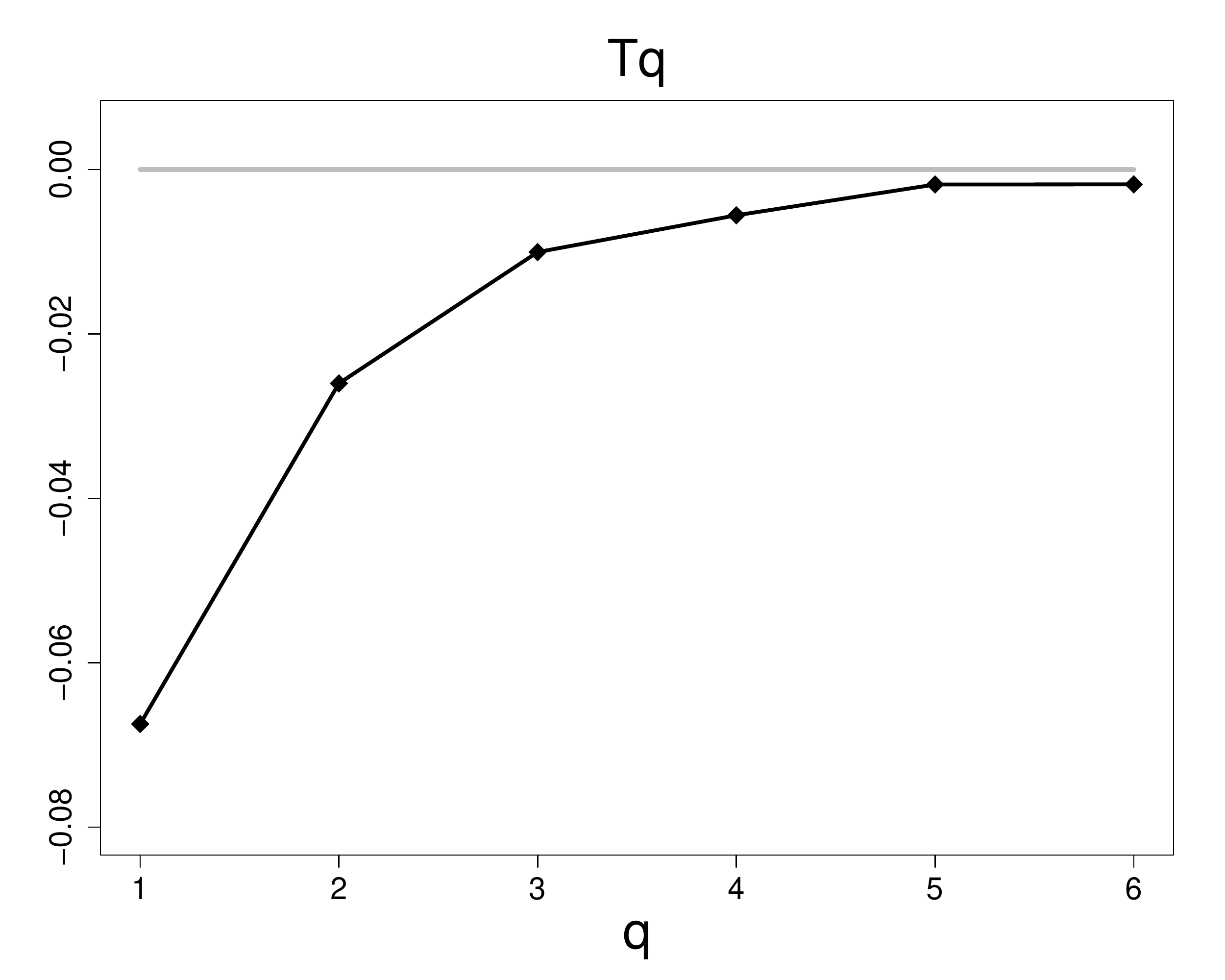}\\
\vspace{0.3cm}\hspace{-6cm}(c) &\hspace{-3.5cm}(d) \\
\includegraphics[width=0.45\textwidth]{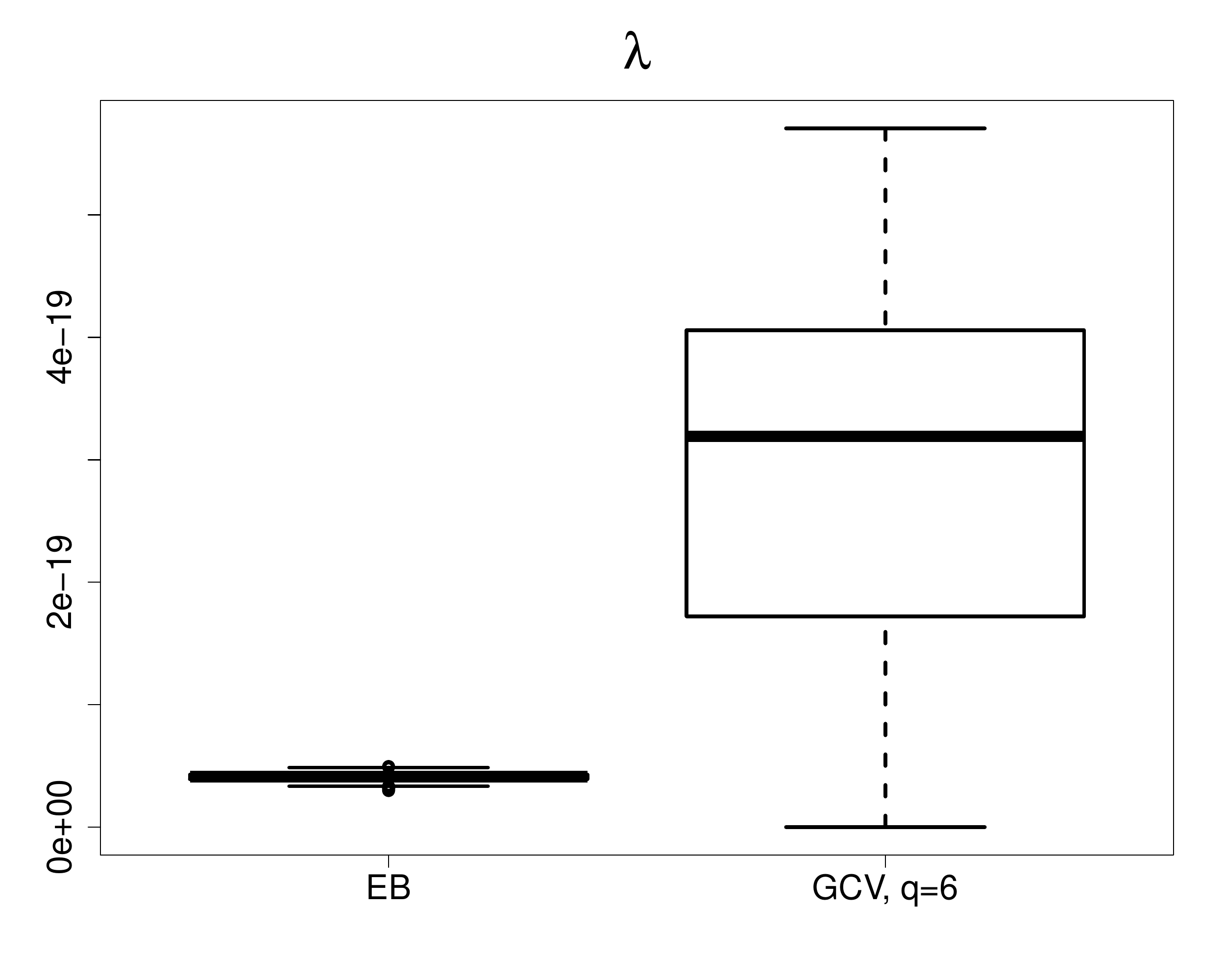}
\includegraphics[width=0.45\textwidth]{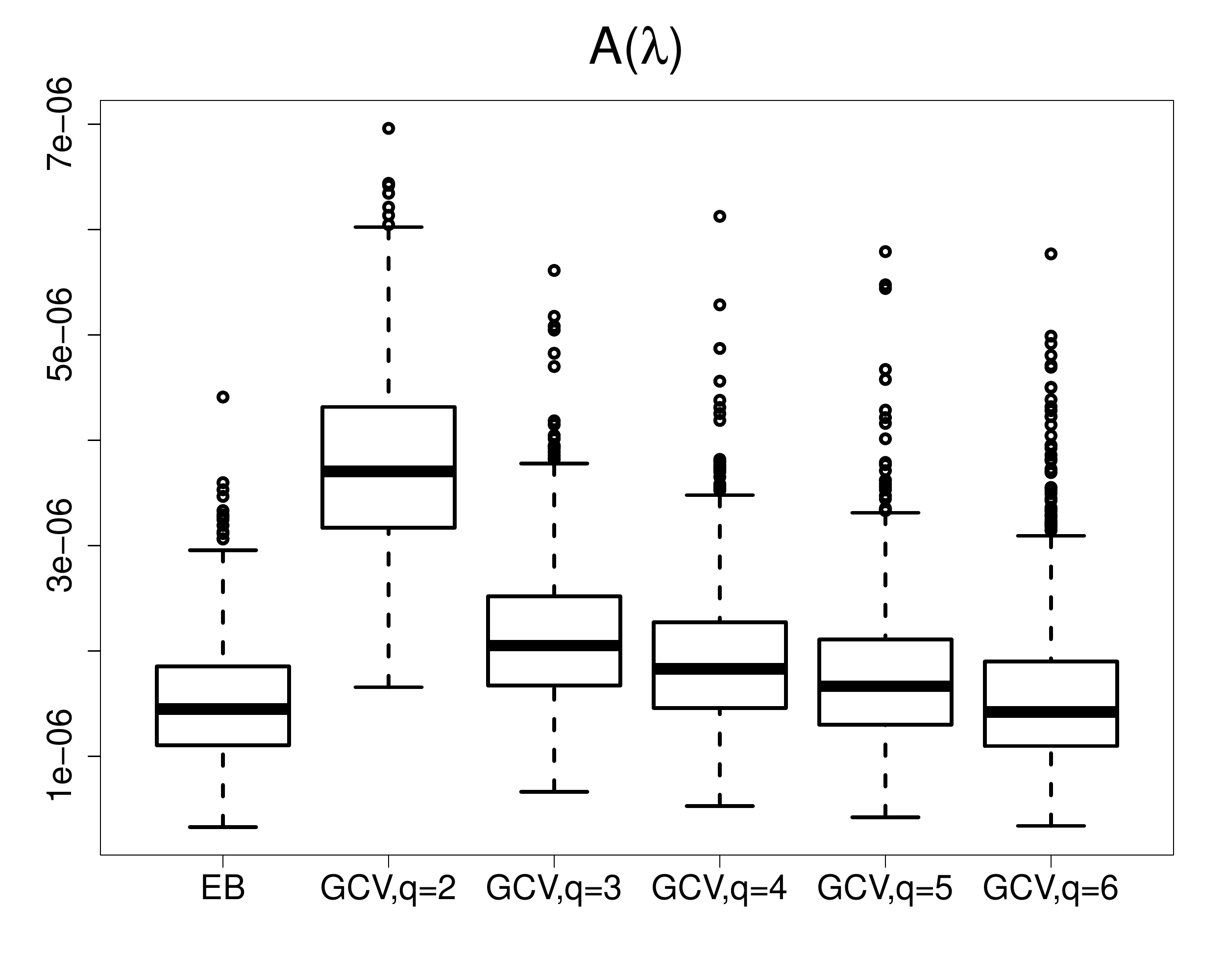}
\end{tabular}
\caption{(a) Function $f_2$, (b) Criterion $T_q(\hat\lambda_q,q)$ for
  $f_2$, (c) Boxplots of $\hat\lambda_{\hat{q}}$ and $\hat\lambda_f$ obtained with GCV and
  $q=6$, (d) Boxplots of $A(\hat\lambda_{\hat{q}})$ (EB) and
  $A(\hat\lambda_f)$}
\label{figure:f2}
\end{figure}

Now we consider the simulation results for the analytical function
$f_2(x)=\cos(5\pi x)$ which is known to have exponentially decaying
Demmler-Reinsch coefficients. In this setting, criterion
$T_q(\hat\lambda_q,q)$ should remain negative, asymptotically
approaching zero from below, which is visible in the top right plot of
Figure \ref{figure:f2}. We estimated $\hat\lambda_{\hat{q}}$ setting $\hat{q}=6$ if
$T_q(\hat\lambda_q,q)<0$ for $q=6$. For function $f_2$ estimator
$\hat{q}$ is slightly more variable, resulting in $\hat{q}=6$ in $986$
cases and in $\hat{q}=5$ in $14$ cases. This can be due to the fact that in small samples it is difficult
to distinguish between an exponential decay $\exp(-\pi i)$ and a decay
with $(\pi i)^{-q}$, $q>5$, $i=1,\ldots,n$. Also, $\hat\lambda_f$ for
$q=2,\ldots,6$ with generalized cross-validation have been calculated. It appears that
$\hat\lambda_{\hat{q}}$ and $\hat\lambda_f$ with $q=6$ have the same rate, but
again, $\hat\lambda_f$ is much more variable. For these smoothing
parameters we also observe, that the corresponding AMSE ratio is
closest to one. In general, for this function $f_2$ the adaptive empirical
Bayesian smoothing spline estimator again outperforms frequentist splines
uniformly in $q$ with the largest AMSE ratio of about $2.5$ for $q=2$.

Finally, we remark on the implementation of the procedure. It
is well-known that the spline based basis with knots at observations becomes
numerically unstable for higher $q$s. In fact, it seems impossible to get numerically
stable Demmler-Reinsch basis for the natural spline space for
$q>3$ with usual approaches. Instead, we relied on an approximation
based on the Demmler-Reinsch basis (\ref{eq:def_demmler_basis_functions1}) for
${\cal{W}}_q$. The details of this approach will be
reported elsewhere, but the implementation in R is available
from the authors on request.
\section[Conclusions]{Conclusions}\label{sec:conclusions}%

The selection of the order of smoothing splines in non-parametric regression is a topic mostly absent from the literature.
The empirical Bayes method is shown to provide an adequate framework to produce data driven choices for this parameter.
Although the dependence of the prior on the parameter $q$ -- which
controls the order of the smoothing spline -- is rather implicit, if
the regression function has a well defined smoothness (determined with
the help of so the so called polished-tail condition), then $\hat q$ is consistent
and we identify the smallest Sobolev space containing the regression function.
Hence, our adaptive empirical Bayesian smoothing spline estimator (which is the mean of the empirical Bayes posterior) adapts to the underlying smoothness of the signal.

High probability regions of the empirical posterior are shown to have good frequentist coverage properties.
For a large class of functions the size of these regions is shown to adapt to the underlying smoothness of the signal, effectively quantifying the amount of uncertainty of the empirical Bayesian smoothing spline estimator.
These results are used to show that frequentist smoothing splines are outperformed by empirical Bayesian smoothing splines.


\section*{Acknowledgements}
We would like to thank the editor, the associate editor and both
referees for the extremely useful remarks, that helped to improve the
paper substantially. We express also our gratitude to Francisco
Rosales for his implementation of the Demmler-Reinsch basis and
helpful discussions.

\section[Technical details]{Technical results and proofs}\label{appendix:technical_results}

\subsection{Demmler-Reinsch basis and estimating equations}\label{appendix:DR_basis}

Let $\{\psi_i\}_{i=1}^\infty$ denote the Demmler-Reinsch basis of
${\cal{W}}_\beta(M)$, such that it holds
\begin{equation}\label{eq:def_demmler_basis_functions1}
\int_0^1\psi_{\beta,i}(x)\psi_{\beta,j}(x)dx=\delta_{ij}=\nu_{\beta,i}^{-1}\int_0^1\psi_{\beta,i}^{(\beta)}(x)\psi_{\beta,j}^{(\beta)}(x)dx.
\end{equation}
Hence, any $f\in{\cal{W}}_\beta(M)$ can be represented as 
$f=\sum_{i=1}^\infty \theta_{\beta,i}\psi_{\beta,i}$ for
$\theta_{\beta,i}=\int_0^1f(x)\psi_{\beta,i}(x)dx$, and
$\|f^{(\beta)}\|^2=\sum_{i=1}^\infty \theta_{\beta,i}^2\nu_{\beta,i}<M^2$. We also
define Sobolev spaces for real-valued $\beta$s via
\beq
\label{eq:genSobolev}
{\cal{W}}_\beta(M)=\{f:\;f\in {\cal{C}}^{\lfloor\beta\rfloor-1}[0,1],\;\|f^{(\beta)}\|^2=\sum_{i=1}^\infty \theta_{\beta,i}^2\nu_{\beta,i}^{\beta/\lfloor\beta\rfloor}<M^2\}.
\eeq

The Demmler-Reinsch basis of the natural spline space of degree $2q-1$
with knots at
observations ${\cal{S}}_{2q-1}(\bm{x})$ is defined via
\begin{equation}\label{eq:def_demmler_basis_functions2}
\sum_{k=1}^n\phi_{q,i}(x_k)\phi_{q,j}(x_k)=\delta_{ij}=\eta_{q,i}^{-1}\int_0^1\phi_{q,i}^{(q)}(x)\phi_{q,j}^{(q)}(x)dx.
\end{equation}
With the Demmler-Reinsch basis the smoother matrix can be represented as
$
\bm{S}_{\lambda,q} = \bm{\Phi}_q\big\{\bm{I}_n+\lambda
n\diag(\bm{\eta}_q)\big\}^{-1} \bm{\Phi}_q^T,
$
where $\bm{\Phi}_q=\bm{\Phi}_q(\bm{x}) = [\phi_{q,1}(\bm{x}), \dots,
\phi_{q,n}(\bm{x})] =
[\phi_{q,j}(x_i)]_{i,j=1}^n$. For the eigenvalues $\bm{\eta}_q$, Theorem 2.2. in \cite{Speckman:1985} establishes the following
approximation
\begin{equation}\label{eq:diagonalisation}
\eta_{q,1}=0,\; \dots,\; \eta_{q,q}=0,\quad n\eta_{q,i}=\pi^{2q}
(i-q)^{2q}\{1+o(1)\},\; i=q+1,\dots,n,
\end{equation}
where the $o(1)$ term is uniform over $i=o\{n^{2/(2q+1)}\}$, as $n\rightarrow\infty$.

For $q=\beta$ 
in \cite{Speckman:1985} is shown that
$\sqrt{n}\phi_{\beta,i}=\psi_{\beta,i}\{1+o(1)\}$, $n\eta_{\beta,i}=\nu_{\beta,i}\{1+o(1)\}$,
$n^{-1}B_{\beta,i}^2=\theta_{\beta,i}^2\{1+o(1)\}$, as well as $n^{-1}\sum_{i=\beta+1}^nB_{\beta,i}^2n\eta_{\beta,i}=
\|f^{(\beta)}\|^2\{1+o(1)\}$, where
$B_{\beta,i}=\sum_{j=1}^n\phi_{\beta,j}(x_i)f(x_i)$. 
In case $\beta\neq q$, that is
$f\in{\cal{W}}_\beta(M)$, but the Demmler-Reinsch basis of degree $q$
is used, we find in a similar fashion
\beqn
\frac{1}{n} \sum_{i=q+1}^n {B_{q,i}^2 (n \eta_{q,i})^{\min(\beta,q)/q}
}=\|f^{(\min(\beta,q))}\|^2\{1+o(1)\}.
\eeqn

The estimating equation $T_\lambda(\lambda,q)$ defined in Section~\ref{sec:empirical_bayes} is derived in the same way as those in \cite{Krivobokova:2013} using the fact that $d\bm{S}_{\lambda,q}/d\lambda=-\lambda^{-1}(\bm{S}_{\lambda,q}-\bm{S}_{\lambda,q}^2)$; $T_q(\lambda, q)$ follows in the same way with a few extra remarks.
Firstly, to determine the derivatives of eigenvalues $\eta_{q,i}$ with
respect to $q$ we employ
representation~\eqref{eq:diagonalisation}. Secondly, we use that
$\sum_{i=1}^nX_{q,i}^2=\bm{Y}^t\bm{\Phi}_q\bm{\Phi}_q^t\bm{Y}=\sum_{i=1}^nY_i^2$,
which is independent of $q$ and implies that the contribution of
$\partial X_i^2/\partial q$ is negligible. This can be seen as follows. 
For $j=j(n)\propto
\lambda^{-\alpha/(2q)}$ with any $\alpha\in(0,1)$, so that $\lambda n\eta_{j}=o(1)$, similar to the
proof of Lemma 2 in \cite{Krivobokova:2013} follows
\beqn
\frac{1}{n}\sum_{i=1}^n\frac{\partial X_{q,i}^2}{\partial q}\frac{\lambda
  n\eta_{q,i}}{1+\lambda n\eta_{q,i}}&=&\frac{1}{n}\sum_{i=1}^n\frac{\partial
  X_{q,i}^2}{\partial q}-\frac{1}{n}\sum_{i=1}^n\frac{\partial
  X_{q,i}^2}{\partial q}\frac{1}{1+\lambda n\eta_{q,i}}\\
&=&\frac{1}{n}\sum_{i=j+1}^n\frac{\partial
  X_{q,i}^2}{\partial q}\frac{\lambda n\eta_{q,i}}{1+\lambda n\eta_{q,i}}\{1+o(1)\},
\eeqn
that is, $\partial X_{q,i}^2/\partial q$ contributes only to the tail
part of the sum. Next, observe that
$X_{q,i}/\sqrt{n}=B_{q,i}/\sqrt{n}+O_p(\sigma/\sqrt{n})$, so that
\beqn
\frac{1}{n}\sum_{i=j+1}^n\frac{\partial
  X_{q,i}^2}{\partial q}\frac{\lambda n\eta_{q,i}}{1+\lambda
  n\eta_{q,i}}&=&\frac{2\sigma}{\sqrt{n}}\sum_{i=j+1}^n\frac{\partial
  B_{q,i}/\sqrt{n}}{\partial q}\left\{O_p(1)+ B_{q,i}/\sigma\right\}\frac{\lambda n\eta_{q,i}}{1+\lambda
  n\eta_{q,i}}.
\eeqn
Since $n^{-1}\sum_{i=1}^nB_{q,i}^2n\eta_{q,i}<\infty$, we can bound
$|B_{q,i}|/\sqrt{n}\leq
{\mbox{const}}(i-q)^{-(2q+1+\epsilon)/2}$ for some $\epsilon>0$ and for
$i>j$ it follows that $B_{q,i}=O(1)$ for
$\lambda=O\left(n^{-2q/\{\alpha(2q+1+\epsilon)\}}\right)$, which
includes $\lambda_q$ for a suitable choice of
$\alpha\in(0,1)$. Since $(\lambda n\eta_{q,i})^{\delta}/(1+\lambda
n\eta_{q,i})\leq 1$, for some $\delta\in[0,1]$, 
showing that $|\partial B_{q,i}/\partial q|=O\{\log(n)\}|B_{q,i}|$
will imply that 
\beqn
\left|\frac{1}{n}\sum_{i=j+1}^n\frac{\partial
  X_{q,i}^2}{\partial q}\frac{\lambda n\eta_{q,i}}{1+\lambda
  n\eta_{q,i}}\right|&\leq&\frac{2\sigma\log(n)}{\sqrt{
  n}}\sum_{i=j+1}^n\left|\frac{B_{q,i}}{\sqrt{n}}\right|(\lambda n\eta_{q,i})^{{\min(\beta,q)}/{2q}}O_p(1)\\
&=&\frac{\lambda^{\min(\beta,q)/q}\log(n)}{\sqrt{\lambda^{\min(\beta,q)/q} n}}o_p(1).
\eeqn
To see that $|\partial B_{q,i}/\partial q|=O\{\log(n)\} |B_{q,i}|$, note that
for $i>q$
\beqn
B_{q,i}=\sum_{j=1}^nf(x_j)\phi_{q,i}(x_j)=\sum_{j=1}^nf(x_j)\sum_{k<j}\alpha_{i,k}(x_j-x_k)^{2q-1},
\eeqn
for some suitable $\alpha_{i,k}$, $k<j$, since $\phi_i(x_j)$ is a spline
function of degree $2q-1$, which is orthogonal to a polynomial
space. Taking the derivative with respect to $q$, gives the result. 
 We conclude that
the contribution of $\partial X_{q,i}^2/\partial q$ is negligible and
therefore subindices $q$ in $X_{q,i}$ and $B_{q,i}$ will be
subsequently omitted.
\subsection{Auxiliary lemmas}\label{appendix:lemmas}

\begin{lemma}\label{lemma:trace}
Let $\bm{S}_{\lambda,q}$ be the smoother matrix.
Then, for all $q>1/2$, $l\in\mathbb{N}$, $m\in\mathbb{N}\cup\{0\}$, and $\lambda$ such that $\lambda\to0$, $n\lambda\to\infty$, as $n\to\infty$,
\[
\tr\big\{(\bm{I}_n-\bm{S}_{\lambda,q})^m\bm{S}_{\lambda,q}^l\big\} =
\sum_{i=1}^n \frac{(\lambda n\eta_{q,i})^m}{\big(1+\lambda n\eta_{q,i}\big)^{m+l}} =
\lambda^{-1/(2q)} \kappa_q\big(m,l\big)\{1+o(1)\},
\]
where the $o(1)$ term depends only on $m$ and $l$, and the constants $\kappa_q\big(m,l\big)$ are
\begin{equation}\label{eq:trace_constant}
\kappa_q\big(m,l\big) = \frac{\Gamma\{m+1/(2q)\}\, \Gamma\{l-1/(2q)\}}{2\pi q\, \Gamma(l+m)}\cdot
\end{equation}
\end{lemma}
\begin{proof}
See \cite{Krivobokova:2013}.
\end{proof}

\begin{lemma}\label{lemma:trace_2}
Consider $\bm{\eta}_{q,i}$ as in~\eqref{eq:diagonalisation}.
Then, as $n\to\infty$ such that $\lambda\to0$ for all  $m\in\mathbb{N}\cup\{0\}$, $r,l\in\mathbb{N}$, $q>1/2$
\beqn
\sum_{i=q+1}^n \frac{(\lambda n\eta_{q,i})^m \{\log(n\eta_{q,i})\}^r}{\big(1+\lambda n\eta_{q,i}\big)^{m+l}}
=\lambda^{-1/(2q)}\{\log(1/\lambda)\}^r\kappa_q(m,l)\{1+o(1)\}.
\eeqn
\end{lemma}
\begin{proof}
Follows the same lines as Lemma \ref{lemma:trace}.
\end{proof}

\begin{lemma}\label{lemma:quadratic_forms}
Let $m\in\mathbb{N}$, $q>1/2$, and $\lambda\rightarrow0$ such that $\lambda n\rightarrow\infty$.
If $f\in\mathcal{W}_q$,
\beq
\label{eq:quadratic_form_low_q}
\frac1n \sum_{i=q+1}^n \frac{B_i^2 \lambda n\eta_{q,i}
\log(\lambda n\eta_{q,i})}{\big(1+\lambda n\eta_{q,i}\big)^m}
= -\lambda\log(1/\lambda)\|f^{(q)}\|^2\{1+o(1)\}.
\eeq
If $f\in\mathcal{W}_\beta\cap\mathcal{M}$, but $f\not\in\mathcal{W}_{\beta+\epsilon}$, $\epsilon>0$, and $q>\beta>1/2$, then 
\beq
\label{eq:quadratic_form_up_q}
\frac1n \sum_{i=q+1}^n \frac{B_i^2 \lambda n\eta_{q,i}
\log(\lambda n\eta_{q,i})}{\big(1+\lambda n\eta_{q,i}\big)^m} \ge c \lambda^{\beta/q},
\eeq
\beq
\label{eq:quadratic_form_up_q_2}
\frac1n \sum_{i=q+1}^n \frac{B_i^2 \lambda n\eta_{q,i}}{\big(1+\lambda n\eta_{q,i}\big)^m} \ge c \lambda^{\beta/q},
\eeq
for some sufficiently small constant $c>0$.
If $f$ is a polynomial of degree $d-1$ (such that $f\in\mathcal{P}_{d}$ and $f\not\in\mathcal{P}_{d-1}$) and $q\ge d$ then
\beq
\label{eq:quadratic_form_up_q_poly}
\frac1n \sum_{i=q+1}^n \frac{B_i^2 \lambda n\eta_{q,i}
\log(\lambda n\eta_{q,i})}{\big(1+\lambda n\eta_{q,i}\big)^m} =0.
\eeq
\end{lemma}

\begin{proof}
We remind that if $f\in\mathcal{W}_q$ then
\[
\frac1n \sum_{i=q+1}^n B_i^2 n \eta_{q,i} = \|f^{(q)}\|^2\{1+o(1)\}<\infty.
\]
Obviously, if $f$ is a polynomial of degree $d-1$, $d\in\mathbb{N}$ ($f\in\mathcal{P}_{d}$, $f\not\in\mathcal{P}_{d-1}$), then $\|f^{(q)}\|=0$, $q\ge d$ and~\eqref{eq:quadratic_form_up_q_poly} follows.

To prove \eqref{eq:quadratic_form_low_q} let $j=j_n(\alpha)$ be like in the proof of Lemma 2 in \cite{Krivobokova:2013}; more specifically, take $j=\pi^{-1}\lambda^{-\alpha/(2q)}$ for $\alpha\in(0,1)$, so that $\lambda n\eta_{q,j}=o(1)$;
following the proof of that lemma, for each $\alpha\in(0,1)$,
\[
\frac1n \sum_{i=q+1}^n \frac{B_i^2 \lambda n\eta_{q,i} \log(\lambda n\eta_{q,i})}{\big(1+\lambda n\eta_{q,i}\big)^m} = 
-\frac\lambda{n} \Big[\sum_{i=q+1}^j B_i^2  n\eta_{q,i}\log\frac1{\lambda n\eta_{q,i}} \{1+o(1)\} -
O\Big(\sum_{i=j+1}^n B_i^2 n\eta_{q,i}\Big) \Big].
\]
If $f\in\mathcal{W}_q$ such that $n^{-1}{\sum}_{i=q+1}^n B_i^2 n\eta_{q,i} = \|f^{(q)}\|^2\{1+o(1)\} <\infty$, whence for any choice of $j\to\infty$ the second term on the right-hand-side is negligible compared to the first.
Conclude that for each $\alpha\in(0,1)$,
\[
\lambda\log\lambda \|f^{(q)}\|^2\{1+o(1)\} \le
\!\frac1n\! \sum_{i=q+1}^n\!\! \frac{B_i^2 \lambda n\eta_{q,i} \log(\lambda n\eta_{q,i})}{\big(1+\lambda n\eta_{q,i}\big)^m} \le
(1-\alpha)\lambda\log\lambda \|f^{(q)}\|^2\{1+o(1)\},
\]
(for an $o(1)$ term which is uniform over $\alpha$) such that~\eqref{eq:quadratic_form_low_q} follows by taking $\alpha\to0$.

Finally, we establish~\eqref{eq:quadratic_form_up_q}.
Let $f\in\mathcal{W}_\beta(M)$, $f\not\in\mathcal{W}_{\beta+\epsilon}$, $\epsilon>0$, $q>\beta+\epsilon$.
Consider the function
\[
g_r(x) = \frac{x^r \log(x)}{(1+x)^m}, \quad r\in(0,1),\; x>0.
\]
This continuous function is negative on the interval $(0,1)$ and positive on $(1,\infty)$.
It converges to zero when $x\to0$, has a global minimum at some point $m_r\in(0,1)$, is zero at $x=1$, has a global maximum at $M_r\in(1,\infty)$, and converges monotonically to zero as $x\to\infty$.
Let $j_{-2}\le j_{-1} \le j_0 \le j_1 \le j_2$ be sequences of integers depending on $n$ such that:
$g_{1-\beta/q}(\lambda n\eta_{q,j_{-2}}) = o(1)$ and
$g_{1-\beta/q}(\lambda n\eta_{q,j_{-1}}) = o(1)$ with $\lambda n\eta_{q,j_{-2}}<m_{1-\beta/q}<\lambda n\eta_{q,j_{-1}}$ and $j_{-2}\to\infty$;
$\lambda n\eta_{q,j_0} < 1 \le \lambda n\eta_{q,j_0+1}$; and for some (small) constant $0<c<g_1(M_1)$,
$g_1(\lambda n\eta_{q,i})\ge c$, $i=j_1, \dots, j_2$, such that $g_1(\lambda n\eta_{q,j_1-1})<c$, $g_1(\lambda n\eta_{q,j_2+1})<c$.
(Note that $m_r$, $M_r$, and the sequences $j_k$, $k=-2\dots,2$ depend on $r$ and $m$.)

Splitting the summation along these sequences, and using the bounds above
\begin{align*}
&\frac1n \sum_{i=q+1}^{j_0} \frac{B_i^2 \lambda n\eta_{q,i} \log(\lambda n\eta_{q,i})}{\big(1+\lambda n\eta_{q,i}\big)^m} \ge\\&\quad\quad\ge
-\lambda^{\beta/q}\frac1n\Big\{o(1)\sum_{i\in I} B_i^2 (n\eta_{q,i})^{\beta/q}
-g_{1-\beta/q}(m_{1-\beta/q})\sum_{i=j_{-2}+1}^{j_{-1}} B_i^2 (\lambda n\eta_{q,i})^{\beta/q} \Big\},
\end{align*}
where $I=\{q+1,\dots,j_{-2}\}\cup\{j_{-1},\dots,j_0\}$.
Since by definition $f\in\mathcal{W}_{\beta}$, then we have $n^{-1}{\sum}_{i=q+1}^nB_i^2 (\lambda n\eta_{q,i})^{\beta/q} = \|f^{(\beta)}\|^2\{1+o(1)\}<\infty$, so, since $j_{-2}\to\infty$ we conclude that the lower bound in the previous display is $- o(\lambda^{\beta/q})\|f^{(\beta)}\|^2$, which lower bounds the sum of the negative terms.
As for the sum of the positive terms,
\[
\frac1n \sum_{i=j_0+1}^n \frac{B_i^2 \lambda n\eta_{q,i} \log(\lambda n\eta_{q,i})}{\big(1+\lambda n\eta_{q,i}\big)^m} \ge
\frac1n \sum_{i=j_1}^{j_2} \frac{B_i^2 \lambda n\eta_{q,i} \log(\lambda n\eta_{q,i})}{\big(1+\lambda n\eta_{q,i}\big)^m} \ge
\frac cn \sum_{i=j_1}^{j_2} B_i^2.
\]
Simple computations show we can take $O(\lambda^{-1/(2q)})=j_1\le j_2/\rho$, for any $\rho>2$.
By the polished-tail condition the previous display can further be lower bounded by
\[
\frac cn\sum_{i=j_1}^{\rho j_1} B_i^2 \ge
\frac {c\, (n\eta_{q,\rho j_1})^{-\frac{\beta+\epsilon}q}}n \sum_{i=j_1}^{\rho j_1}B_i^2(n\eta_{q,i})^{\frac{\beta+\epsilon}q} \ge
\frac {O(j_1^{-2(\beta+\epsilon)})}{n L} \sum_{i=j_1}^{n}B_i^2(n\eta_{q,i})^{\frac{\beta+\epsilon}q}.
\]
The sum ${\sum}_{i=j_1}^{n}B_i^2(n\eta_{q,i})^{\frac{\beta+\epsilon}q}/n$ converges to infinity as long as $j_1=o(n)$, $\epsilon>0$.
The power $j_1^{-2(\beta+\epsilon)}$ is $o(\lambda^{\beta/q})$ for every $\epsilon>0$ but it is $O(\lambda^{\beta/q})$ if $\epsilon=0$.
Conclude that for small enough $\epsilon$, namely such that $j_1^{-2(\beta+\epsilon)}$ is of a larger order than the $o(\lambda^{\beta/q})$ term from the bound on the negative terms, the sum of the positive terms dominates.
In fact, picking an appropriate sequence $\epsilon\to0$ one can show that the sum of the positive terms is lower bounded by a small enough multiple of $\lambda^{\beta/q}.$
The bound in~\eqref{eq:quadratic_form_up_q} follows.
The bound in~\eqref{eq:quadratic_form_up_q_2} follows in the same way by adjusting the definition of the sequences $j_1$ and $j_2$.
\end{proof}

\begin{lemma}[Lemma 3 from~\citealp{Krivobokova:2013}]\label{lemma:quad_trace_at_oracle}
Let $\lambda\to0$ be such that $n\lambda\to\infty$.
Then, for any $l\in\mathbb{N}$,
\begin{align*}
\bm{f}^T(\bm I_n-\bm S_{\lambda,q})\bm S_{\lambda,q}^{1+l}\bm f|_{\lambda=\lambda_q} &=
\sigma^2 \tr(\bm S_{\lambda,q}^2)|_{\lambda=\lambda_q}\{1+o(1)\},\\
\bm{f}^T(\bm I_n-\bm S_{\lambda,q})^{1+l}\bm S_{\lambda,q}\bm f|_{\lambda=\lambda_q} &=
o(\lambda_q^{-1/(2q)}).
\end{align*}
\end{lemma}

\begin{lemma}\label{lemma:variance}
Let $T_{C_p}(\lambda)$ be the estimating equation for $\lambda_f$ as defined in~\cite{Krivobokova:2013}, $\lambda_f$ the solution to $\mathbb{E}T_{C_p}(\lambda)=0$, and $\hat\lambda_f$ the solution to $T_{C_p}(\lambda)=0$.
Then
\[
\mathbb{E}\big\{\lambda_f^{-1/(4q)}(\hat\lambda_f/\lambda_f-1)\big\}^2 \le
\frac{4\kappa_q(4,2)}{\big\{ 4\kappa_q(1,2)-3\kappa_q(1,3) \big\}^2}\{1+o(1)\},
\]
such that in particular $\mathbb{V}\big(\hat\lambda_f/\lambda_f-1\big) = o(1)$.
\end{lemma}

\begin{proof}
By definition of $\hat\lambda_f$, for some $\tilde\lambda$ between $\hat\lambda_f$ and $\lambda_f$,
\[
0 = T_{C_p}(\hat\lambda_f) = T_{C_p}(\lambda_f) + T_{C_p}'(\tilde\lambda) \big(\hat\lambda_f-\lambda_f\big).
\]
It is known (cf.\ the supplementary materials to~\citealp{Krivobokova:2013}) that,
\begin{align*}
\mathbb{V}\big\{T_{C_p}(\lambda_f)\big\}		&= 2\frac{\sigma^4}{n^2} \kappa_q(4,2) \lambda_f^{-1/(2q)}\{1+o(1)\},\\
\mathbb{E}\big\{T_{C_p}'(\tau\lambda_f)\big\}	&= \mathbb{E}\big\{T_{C_p}'(\lambda_f)\big\} \Big\{ 1 + \frac{2q-1}{2q}O(\tau-1) \Big\}\\
									&=
\frac{\sigma^2}n \lambda_f^{-1-1/(2q)}\big(4\kappa_q(1,2)-3\kappa_q(1,3)\big) \Big\{ 1 + \frac{2q-1}{2q}O(\tau-1) \Big\},
\end{align*}
for $\tau\in[1-\epsilon, 1+\epsilon]$ and small enough $\epsilon>0$.
Let $A(\epsilon)$ be the event $\{\big|\hat\lambda_f/\lambda_f-1\big|<\epsilon\}$, and write $\tilde\lambda=\tau\lambda_f$, so that
\begin{align*}
&2\frac{\sigma^4}{n^2} \kappa_q(4,2) \lambda_f^{-1/(2q)}\{1+o(1)\} =\\ & \quad\quad=
\mathbb{E}\Big\{\big( \hat\lambda_f-\lambda_f \big)^2\mathbb{E}\big[ T_{C_p}'(\tilde\lambda)^2\big| \hat\lambda_f \big]\Big\} \ge
\mathbb{E}\Big\{\big( \hat\lambda_f-\lambda_f \big)^2\mathbb{E}\big[ T_{C_p}'(\tau\lambda_f)^2\big| \hat\lambda_f, A(\epsilon) \big]\mathbb{P}\big(A(\epsilon)\big)\Big\}\\ &\quad\quad =
\mathbb{E}\Big\{\big( \hat\lambda_f-\lambda_f \big)^2\mathbb{E}\big[ T_{C_p}'(\tau\lambda_f)\big|\,A(\epsilon) \big]^2\{1+o(1)\}\Big\}\\ &\quad\quad=
\mathbb{E}\Big\{\big( \hat\lambda_f-\lambda_f \big)^2\mathbb{E}\big[ T_{C_p}'(\lambda_f)\big]^2\mathbb{E}\big[ 1 + \frac{2q-1}{2q}O(\tau-1)\big|\,A(\epsilon) \big]^2\{1+o(1)\}\Big\},
\end{align*}
where we use the fact that by definition $\mathbb{E}\big\{T_{C_p}(\lambda_f)\big\}=0$, and Jensen's inequality.
Since the previous display holds for all small enough $\epsilon$, set $\epsilon$ so that $1 + (2q-1)O(\tau-1)/(2q)\ge 1/\sqrt{2}$, say, whence
\[
\mathbb{E}\big\{\hat\lambda_f/\lambda_f-1\big\}^2 \le
\lambda_f^{1/(2q)}\frac{4\kappa_q(4,2)}{\big\{ 4\kappa_q(1,2)-3\kappa_q(1,3) \big\}^2}\{1+o(1)\} = o(1).
\]
\end{proof}

\subsection[Proof of Theorem~\ref{theorem:consistency_q}]{Proof of Theorem~\ref{theorem:consistency_q}}\label{appendix:proof_consistency_q}

Assume $f\in L_2$.
First, note that for some $\tilde\lambda$ between $\hat\lambda_q$ and $\lambda_q$, a.s.,
\begin{equation}\label{eq:variance_q}
T_q(\hat\lambda_q,q) = T_q(\lambda_q,q) + (\hat\lambda_q-\lambda_q)\left.\frac{\partial T_q(\lambda,q)}{\partial\lambda}\right|_{\lambda=\tilde\lambda}.
\end{equation}
We have
\beqn
{\mathbb{V}}T_q(\lambda_q,q)&=&\frac{2\sigma^2}{q^2n^2}\left[\frac{\{\log(1/\lambda_q)\}^2}{\lambda_q^{1/(2q)}/\kappa_q(2,2)}\sigma^2\{1+o(1)\}+2\sum_{i=q+1}^nB_i^2\left\{\frac{\lambda_qn\eta_{q,i}\log(n\eta_{q,i})}{(1+\lambda_qn\eta_{q,i})^2} \right\}^2 
\right]\\ 
&=&O\left\{\lambda_q^{-1/(2q)}\log(1/\lambda_q)^2n^{-2}\right\}\rightarrow 0,\quad n\rightarrow\infty.
\eeqn

Hence, $T_q(\lambda_q,q)-\mathbb{E}T_q(\lambda_q,q)\stackrel{P}{\longrightarrow}0$ for any $q>1/2$.
In fact, for any $\epsilon>0$, and $\mathbb{Q}_n$ as defined at the end of Section~\ref{sec:empirical_bayes},
\begin{align*}&
\mathbb{P}\Big( \max_{q\in \mathbb{Q}_n}\big|T_q(\lambda_q,q)-\mathbb{E}T_q(\lambda_q,q)\big| >\epsilon \Big) \le 
\mathbb{P}\Big( {\sum}_{q\in \mathbb{Q}_n}\big|T_q(\lambda_q,q)-\mathbb{E}T_q(\lambda_q,q)\big| >\epsilon \Big) \le\\ &
\sum_{q\in \mathbb{Q}_n} \mathbb{P}\Big( \big|T_q(\lambda_q,q)-\mathbb{E}T_q(\lambda_q,q)\big| >\epsilon/|\mathbb{Q}_n| \Big) \le
\frac{|\mathbb{Q}_n|^3w}{\epsilon^2}\max_{q\in \mathbb{Q}_n}\mathbb{V}T_q(\lambda_q,q),
\end{align*}
by Chebyshev's inequality.
We assume without loss of generality that $|\mathbb{Q}_n|= O\{1/\log(n)^\alpha\}$, for some $\alpha>2$ such that the upper-bound in the previous display converges to zero, implying that $T_q(\lambda_q,q)-\mathbb{E}T_q(\lambda_q,q)\stackrel{P}{\longrightarrow}0$ takes place uniformly over $q\in \mathbb{Q}_n$.
Note also that ${\mathbb{V}}T_q(\lambda_q,q) = o\big[\{{\mathbb{E}}T_q(\lambda_q,q)\}^2\big]$.
The remaining term in~\eqref{eq:variance_q} is controlled using the same arguments as Lemma~\ref{lemma:variance} so we omit the details.

Consider a polished tail $f\in\mathcal{L}_2$.
Without loss of generality say that $T_q$, $q\not\in \mathbb{Q}_n$, is obtained by linear interpolation of the values of $T_q$ at $\mathbb{Q}_n$.
From the discussion in Sections~\ref{sec:asymptotics:empirical_bayes_lambda} and~\ref{sec:asymptotics:empirical_bayes_q}, the criterion $\mathbb{E}T_q(\lambda_q,q)$ is strictly negative for $q\le\bar\beta$ and is positive for $q> \bar\beta$, whence
\begin{align*}
&\mathbb{P}\Big\{\max_{q\in \{r\in \mathbb{Q}_n:\, r<\bar\beta\}}T_q(\hat\lambda_q,q)<0, \,
\min_{q\in \{r\in \mathbb{Q}_n:\, r>\bar\beta\}}T_q(\hat\lambda_q,q)>0 \Big\} \le\\ & \quad\quad\quad\quad \le
\mathbb{P}\Big\{
\sup_{q\in(1/2,\, \bar\beta)}T_q(\hat\lambda_q,q)<0, \,
\inf_{q\in(\bar\beta,\, \log(n)]}T_q(\hat\lambda_q,q)>0 \Big\} \le
\mathbb{P}\Big\{\hat q = \bar\beta\Big\},
\end{align*}
The first inequality follows from the fact that we linearly interpolate the criterion, and the second inequality follows from the definition of $\hat q$.
From the computations above, the left-hand-side of the previous display converges to 1 as $n\to\infty$, so that the second statement of the theorem follows.
For $f\in\mathcal{W}_\beta(M)$, $\beta>1/2$, by the same arguments as above,
\begin{align*}
&\mathbb{P}\Big[
\max_{q\in \{r\in \mathbb{Q}_n:\, r<\beta\}} T_q(\hat\lambda_q,q)<0, \,
\min_{q\in \{r\in \mathbb{Q}_n:\, r>\beta\}}T_q(\hat\lambda_q,q)=
o_p\{T_q(\hat\lambda_{\beta},\beta)\} \Big] \le\\ & \quad\le 
\mathbb{P}\Big[
\sup_{q\in(1/2,\, \beta)} T_q(\hat\lambda_q,q)<0, \,
\inf_{q\in(\beta,\, \log(n)]} T_q(\hat\lambda_q,q)=
o_p\{T_q(\hat\lambda_{\beta},\beta)\} \Big] \le 
\mathbb{P}\Big[\beta \le \hat q \le \log(n)\Big],
\end{align*}
and the first statement of the theorem follows since the left-hand-side of the previous display converges to 1, as $n\to\infty$.
The last statement of the theorem follows by using the exact same argument.

\subsection[Proof of Theorem~\ref{theorem:coverage}]{Proof of Theorem~\ref{theorem:coverage}}\label{appendix:coverage}

If we assume that for some countable set $Q_n\in\mathbb{Q}_n$,
\begin{equation}\label{eq:A1}
\inf_{\bm{f}\in\mathcal{W}} \mathbb{P}_f \big( \hat q \in Q_n \big) = 1+o(1),
\end{equation}
then~\eqref{eq:honest_coverage} follows if we show that uniformly over $q\in Q_n$,
\begin{equation}\label{eq:C1}
\inf_{\bm{f}\in\mathcal{W}} \mathbb{P}_f \big\{ \|\bm{f}-\bm{\hat{f}}_{\hat\lambda_{ q}, q}\| \le \hat{\sigma}_{\hat\lambda_{ q}, q}\, L\, r_n(\hat\lambda_{ q},  q) \big\} = 1+o\big(1/|Q_n|\big).
\end{equation}
For each $q\in Q_n$, $f\in\mathcal{W}$, consider intervals $L_q=[\underline{\lambda}_q, \bar\lambda_q]$ -- in fact $L_q(f)=[\underline{\lambda}_q(f), \bar\lambda_q(f)]$ -- with $\lambda_q\in L_q$ and $\bar\lambda_q/\underline{\lambda}_q\to1$, such that
\begin{equation}\label{eq:A2}
\inf_{\bm{f}\in\mathcal{W}} \mathbb{P}_f \big\{ \hat\lambda_q \in L_q(f) \big\} = 1+o(1).
\end{equation}
If there exist constants $C_1,C_2,C_3,C_4$, such that uniformly over $f\in\mathcal{W}$ and $q\in Q_n$,
\begin{align}
\inf_{\lambda\in L_q} r_n^2(\lambda,q) &\ge C_1^2 \frac{\bar\lambda_q^{-1/(2q)}}n, \label{eq:S1}\\
\inf_{f\in\mathcal{W}} \mathbb{P}_f  \big( \inf_{\lambda\in L_q} \hat\sigma_{\lambda,q}^2 \ge C_2^2 \big) &= 1+o(1),\label{eq:S2}\\
\sup_{\lambda\in L_q} \| \mathbb{E}_f \bm{\hat{f}}_{\lambda,q} - \bm{f}\|^2 &\le C_3^2 \frac{\bar\lambda_q^{-1/(2q)}}n,\label{eq:S3}\\
\inf_{f\in\mathcal{W}} \mathbb{P}_f 
\Big( \sup_{\lambda\in L_q} \|\bm{\hat{f}}_{\lambda,q} -\mathbb{E}_f \bm{\hat{f}}_{\lambda,q}\|^2 \le C_4^2 \bar\lambda_q^{-1/(2q)}/n
\Big) &= o\big(1/|Q_n|\big), \label{eq:S4}
\end{align}
then, if $L$ is large enough such that $C_4 + C_3 \le L\, C_2\, C_1$, conclude
\[
\inf_{f\in\mathcal{W}} \mathbb{P}_f \Big\{ \sup_{\lambda\in L_q} \|\bm{\hat{f}}_{\lambda,q} -\mathbb{E}_f \bm{\hat{f}}_{\lambda,q}\| + \sup_{\lambda\in L_q} \|\mathbb{E}_f \bm{\hat{f}}_{\lambda,q}-\bm{f}\| \le \inf_{\lambda\in L_q} \hat\sigma_{\lambda,q} L\, r_n(\lambda,q) \Big\} = 1+o\big(1/|Q_n|\big),
\]
such that by the triangle inequality,~\eqref{eq:C1} holds.\\

If~\eqref{eq:A1} holds, then~\eqref{eq:optimal_radius} follows if we show that uniformly over $q\in Q_n$,
\begin{equation}\label{eq:S5}
\inf_{f\in\mathcal{W}_\beta(M)}\mathbb{P}_f \big\{ r_n(\hat\lambda_q,q) \le Kn^{-\beta/(2\beta+1)} \big\} =
1+ o\big(1/|Q_n|\big).
\end{equation}
Since the risk of the smoothing spline satisfies~\eqref{eq:risk}, it is only possible to establish~\eqref{eq:S5} if $q\ge\beta$.
Consider $s_n(\beta)=\max\{q\in\mathbb{Q}_n: q\le\beta\}$.
By the first and second statements of Theorem~\ref{theorem:consistency_q} we conclude that if we take $Q_n=\mathbb{Q}_n\cap(s_n,\bar\beta+1)$, then $\mathbb{P}\big(\hat q\in Q_n\big)=1+o(1)$, such that every $q\in Q_n$ is such that $q\ge\beta$.
Note also that $\mathbb{Q}_n$ can be taken such that $|Q_n|=o\{\log(n)^2\}$, say.
These facts are used when checking~\eqref{eq:S2}, \eqref{eq:S4}, and~\eqref{eq:S5} below.\\

Condition~\eqref{eq:A2} follows from the consistency results in Section~\ref{sec:asymptotics} so that
it suffices to check conditions~\eqref{eq:S1}--\eqref{eq:S5}, which we do in the remainder of this section. 

\subsubsection[Condition~\eqref{eq:S1}]{Checking condition~\eqref{eq:S1}}\label{appendix:coverage:check_S1}

The sequence $r_n(\lambda,q)$ is implicitly defined by
\[
\mathbb{P}\Big\{ R_n(\lambda,q) \le r_n(\lambda,q)^2 \Big\} = 1-\alpha, \quad\text{with}\quad
R_n(\lambda,q)\sim \frac1N\bm{Z}^T\bm{S}_{\lambda,q}\bm{Z} \sim
\frac1N \sum_{i=q+1}^n\frac{\epsilon_i^2}{1+\lambda n\eta_{q,i}},
\]
where $\bm{Z}\sim N(\bm{0},\bm{I}_n)$ and $N\sim\mathcal{X}_n^2$, with $\bm{Z}$ and $N$ independent, and $\bm{\epsilon}=\bm{\Phi}^T\bm{Z}\sim N(\bm{0},\bm{I}_n)$.
Since $R_n(\lambda,q)$ decreases with $\lambda$, the infimum in~\eqref{eq:S1} is attained at $\lambda=\bar\lambda_q$.
Set $R=R_n(\bar\lambda_q,q)$, $r=r_n(\bar\lambda_q,q)$ and note that by the one-sided Chebychev's inequality, $r^2 \ge \mathbb{E}R - \{(1-\alpha)/\alpha\}^{1/2} (\mathbb{V}R)^{1/2}$.
Further, since $N^{-1}$ is independent of $\bm{\epsilon}$, and has an inverse-chi-squared distribution with $n$ degrees of freedom such that $\mathbb{E}N^{-1}=(n-2)^{-1}$,
\[
\mathbb{E}R =
\mathbb{E}\frac1N\, \sum_{i=q+1}^n\frac{\mathbb{E}\epsilon_i^2}{1+\bar\lambda_q n\eta_{q,i}} =
\frac1{n-2}\tr\big(\bm{S}_{\bar\lambda_q,q}\big) =
\frac{\bar\lambda_q^{-1/(2q)}}{n-2} \kappa_q(0,1)\{1+o(1)\}.
\]
By similar computations, since $\mathbb{V}N^{-1} = 2(n-2)^{-2}(n-4)^{-1}$,
\[
\mathbb{V}R =
3\Big(\mathbb{E}\frac1N\Big)^2 \tr\big(\bm{S}_{\bar\lambda_q,q}^2\big) +
\mathbb{V}\frac1N \tr\big(\bm{S}_{\bar\lambda_q,q}\big)^2 +
3\mathbb{V}\frac1N \tr\big(\bm{S}_{\bar\lambda_q,q}^2\big) =
o\Big\{\big(\mathbb{E}R\big)^2\Big\}.
\]
Conclude that~\eqref{eq:S1} holds for $C_1^2<\kappa_q(0,1)\{1+o(1)\}$.

\subsubsection[Condition~\eqref{eq:S2}]{Checking condition~\eqref{eq:S2}}\label{appendix:coverage:check_S2}

We may take $L_q=[c\bar\lambda_q,\bar\lambda_q/c]$, for some $c = 1+o(1) < 1$.
For $\lambda\in L_q$, we bound, a.s.,
\[
\hat{\sigma}_{\lambda,q}^2 =
\frac1n \bm{Y}^T\big(\bm{I}_n-\bm{S}_{\lambda,q}\big)\bm{Y} \ge
\frac{c^2}n \sum_{i=q+1}^n \frac{X_i^2 \bar\lambda_q n\eta_{q,i}}{1+\bar\lambda_q n\eta_{q,i}} \ge
\frac{c^2}n \bm{Y}^T\big(\bm{I}_n-\bm{S}_{\bar\lambda_q,q}\big)\bm{Y}.
\]
Denote this lower-bound by $S=S_n(\bar\lambda_q,q,c)$.
By Chebyshev's inequality,
\[
\mathbb{P} \Big( \inf_{\lambda\in L_q}\hat{\sigma}_{\lambda,q}^2 \ge C_2^2 \Big) \ge 
1-\mathbb{P} \Big\{ S \le \mathbb{E}S - \big(\mathbb{E}S-C_2^2\big) \Big\} \ge
1- \frac{\mathbb{V}S}{\mathbb{V}S+(\mathbb{E}S-C_2^2)^2}.
\]
Using Lemma~\ref{lemma:trace} we have
\[
\mathbb{E}S =
\frac{c^2}n \Big\{ \sigma^2\tr(\bm{I}_n-\bm{S}_{\bar\lambda_q,q}) + \bm{f}^T(\bm{I}_n-\bm{S}_{\bar\lambda_q,q})\bm{f} \Big\} \ge
c^2\sigma^2 -c^2\sigma^2\frac{\tr(\bm{S}_{\bar\lambda_q,q})}{n} = \sigma^2\{1+o(1)\}.
\]
By Lemma~\ref{lemma:trace}, and the bound on $\bm{f}^T(\bm{I}_n-\bm{S}_{\bar\lambda_q,q})^2\bm{f}$ from the next section,
\[
\mathbb{V}S = \frac1{n^2} \Big[ 2\sigma^4 \tr\big\{(\bm{I}_n-\bm{S}_{\bar\lambda_q,q})^2\big\} + 4\sigma^2 \bm{f}^T(\bm{I}_n-\bm{S}_{\bar\lambda_q,q})^2\bm{f} \Big] = O\Big(\frac{\mathbb{E}S}{n}\Big).
\]
Condition~\eqref{eq:S2} holds by taking $C_2^2 < \sigma^2\{1+o(1)\}$.
Since the bound in previous display is $o\{1/\log(n)^2\}$ we further assume that the bound holds uniformly over $q\in Q_n$.

\subsubsection[Condition~\eqref{eq:S3}]{Checking condition~\eqref{eq:S3}}\label{appendix:coverage:check_S3}

Let $L_q$ be as in the previous section.
For $\lambda\in L_q$ the squared bias term satisfies
\[
\| \mathbb{E} \bm{\hat{f}}_{\lambda,q} - \bm{f}\|^2 =
\frac1n\sum_{i=q+1}^n \frac{B_i^2(\lambda n\eta_{q,i})^2}{(1+\lambda n\eta_{q,i})^2} \le
\frac1n\sum_{i=q+1}^n \frac{B_i^2(\bar\lambda_q n\eta_{q,i})^2}{(1+\bar\lambda_q n\eta_{q,i})^2} \le
\frac1{c^2}\| (\bm{S}_{\lambda_q,q} - \bm{I}_n) \bm{f}\|^2.
\]
By Lemma~\ref{lemma:quad_trace_at_oracle} and simple manipulations,
\[
\| (\bm{S}_{\lambda_q,q} - \bm{I}_n) \bm{f}\|^2 = \frac1n\bm{f}^T(\bm{S}_{\lambda_q,q} - \bm{I}_n)\bm{S}_{\lambda_q,q}\bm{f}
- \frac1n\bm{f}^T(\bm{S}_{\lambda_q,q} - \bm{I}_n)\bm{f} \le \sigma^2\frac{\tr(\bm{S}_{\lambda_q,q})}{n}\{1+o(1)\}.
\]
The statement follows by Lemma~\ref{lemma:trace}, for $C_3^2> \sigma^2\kappa_q(0,1)\{1+o(1)\}$.

\subsubsection[Condition~\eqref{eq:S4}]{Checking condition~\eqref{eq:S4}}\label{appendix:coverage:check_S4}

For $L_q$ as in the previous sections, uniformly over $\lambda\in L_q$ the variance term satisfies, a.s.,
\[
\|\bm{\hat{f}}_{\lambda,q} -\mathbb{E} \bm{\hat{f}}_{\lambda,q}\|^2 =
\|\bm{S}_{\lambda,q}(\bm{Y}-\bm{f})\|^2 =
\frac1n \sum_{i=q+1}^n \frac{(X_i-B_i)^2}{(1+\lambda n\eta_{q,i})^2} \le
\frac1n \sum_{i=q+1}^n \frac{(X_i-B_i)^2}{(1+\underline\lambda_q n\eta_{q,i})^2},
\]
where $\bm{X}-\bm{B}=\bm{\Phi}^T(\bm{Y}-\bm{f})\sim N(\bm{0},\sigma^2\bm{I}_n)$.
Denote the upper-bound in the previous display by $V=V_n(\underline\lambda_q,q)$.
By Lemma~\ref{lemma:trace}, the expectation of $V$ is
\[
\mathbb{E}V = \frac{\sigma^2}n \tr\big(\bm{S}_{\underline\lambda_q,q}^2\big) = 
\frac{\sigma^2}n \underline\lambda_q^{-1/(2q)}\kappa_q(0,2)\{1+o(1)\} =
\sigma^2 c^{-1/q}\kappa_q(0,2)\frac{\bar\lambda_q^{-1/(2q)}}n\{1+o(1)\};
\]
the respective variance is $\mathbb{V}V = 3\sigma^2n^{-2} \tr\big(\bm{S}_{\underline\lambda_q,q}^4\big) = O\big(\mathbb{E}V/n\big)$.
Condition~\eqref{eq:S4} follows by an application of the one-sided Chebyshev's inequality for $C_4^2> \sigma^2 \kappa_q(0,2)\{1+o(1)\}$.
Since the bound in previous display is $o\{1/\log(n)^2\}$ we further assume that the bound holds uniformly over $q\in Q_n$ with an $o(1)$ term that is $o\big(1/|Q_n|\big)$.

\subsubsection[Condition~\eqref{eq:S5}]{Checking condition~\eqref{eq:S5}}\label{appendix:coverage:check_S5}

Note that if~\eqref{eq:A2} holds then it suffices to check
\[
\inf_{f\in\mathcal{W}_\beta(M)}\mathbb{P}_f \Big\{ \sup_{\lambda\in L_q}r_n(\lambda,q) \le Kn^{-\beta/(2\beta+1)} \Big\} = 
\inf_{f\in\mathcal{W}_\beta(M)} 1_{\{D_n(f)\}} = 1+o(1),
\]
where $D_n(f)=\{\sup_{\lambda\in L_q(f)}r_n(\lambda,q) \le Kn^{-\beta/(2\beta+1)}\}$.

Proceeding as in Section~\ref{appendix:coverage:check_S1}, the supremum in $D_n(f)$ is attained at $\underline\lambda_q$ so if we redefine $R=R_n(\underline\lambda_q,q)$ and $r=r_n(\underline\lambda_q,q)$, then by the one-sided Chebyshev's inequality $r^2 \le \mathbb{E}R + \{\alpha/(1-\alpha)\}^{1/2}(\mathbb{V}R)^{1/2}$.
By direct substitution, since $\beta\le q<\bar b$, we have that for sufficiently large $n$
\[
\mathbb{E}R = \frac{\underline\lambda_q^{-1/(2q)}}{n} \kappa_q(0,1)\{1+o(1)\} \le K\, n^{-2\beta/(2\beta+1)}, \quad\text{and}\quad
\mathbb{V}R = o\Big\{\big(\mathbb{E}R\big)^2\Big\},
\]
for constant $K$ depending on $\beta$, $\bar b$, $\sigma^2$, and $M$.
Since the bound on the variance in previous display is $o\{1/\log(n)^2\}$ we further assume that the bound holds uniformly over $q\in Q_n$.
Note that in particular, if $q=\beta$,
\[
\big[\|f^{(\beta)}\|^2/\big\{\sigma^2\kappa_\beta(0,2)\big\}\big]^{1/(2\beta+1)} \le
\big[M^2/\big\{\sigma^2\kappa_\beta(0,2)\big\}\big]^{1/(2\beta+1)} = K,
\]
uniformly over $f\in\mathcal{W}_\beta$.
The statement of the theorem follows.

\subsection[Proof of Theorem~\ref{theorem:GCV}]{Proof of Theorem~\ref{theorem:GCV}}\label{appendix:GCV}

For simplicity, we assume in the following that $\hat{f}_{\lambda,q}$ is independent of $\hat\lambda_f$
(say we make two independent observations at each design point in $\bm{x}$;
this certainly does not impair the performance of the GCV based estimate of $\lambda$, and it considerably simplifies the exposition).

Expanding around the oracle $\lambda_f$ we have, a.s.,
\[
\frac1{(1+n\eta_{q,i}\hat\lambda_f)^{m}} = \frac1{(1+n\eta_{q,i}\lambda_f)^{m}}
-m\frac{n\eta_{q,i}\lambda_f}{(1+n\eta_{q,i}\lambda_f)^{m+1}}\, O\big({\hat\lambda_f}/{\lambda_f}-1\big).
\]
Conclude that
$
\tr\big(\bm{S}_{\hat\lambda_f,q}^m\big) = \tr\big(\bm{S}_{\lambda_f,q}^m\big)
-m\tr\big\{(\bm{I}_n-\bm{S}_{\lambda_f,q})\bm{S}_{\lambda_f,q}^m\big\}\, O\big({\hat\lambda_f}/{\lambda_f}-1\big),
$ 
a.s., and that
$
\bm{f}^T\bm{S}_{\hat\lambda_f,q}^m\bm{f} = \bm{f}^T\bm{S}_{\lambda_f,q}^m\bm{f}
-m\bm{f}^T(\bm{I}_n-\bm{S}_{\lambda_f,q})\bm{S}_{\lambda_f,q}^m\bm{f}\, O\big({\hat\lambda_f}/{\lambda_f}-1\big),
$ 
a.s..

We now determine the asymptotic expectation and variance of risk of $\hat{f}_{\hat\lambda_f,q}$.
From $n\|\bm{\hat f}_{\hat\lambda_f,q}-\bm{f}\|^2 = \bm{Y}^T\bm{S}_{\hat\lambda_f,q}^2\bm{Y} - 2\bm{Y}^T\bm{S}_{\hat\lambda_f,q}\bm{f}+\bm{f}^T\bm{f}$, and the expansions above
\begin{align*}
\mathbb{E}\Big[ n\|\bm{\hat f}_{\hat\lambda_f,q}-\bm{f}\|^2 \mid \hat\lambda_f\Big] &=
\sigma^2 \tr\big( \bm{S}_{\hat\lambda_f,q}^2 \big) + \bm{f}^T\big(\bm{I}_n-\bm{S}_{\hat\lambda_f,q}\big)^2\bm{f} =\\ & = 
\Big\{\sigma^2 \tr\big( \bm{S}_{\lambda_f,q}^2 \big) + \bm{f}^T\big(\bm{I}_n-\bm{S}_{\lambda_f,q}\big)^2\bm{f}\Big\}
\Big\{1+O\Big(\frac{\hat\lambda_f}{\lambda_f}-1\Big)\Big\},\\
\mathbb{V}\Big[ n\|\bm{\hat f}_{\hat\lambda_f,q}-\bm{f}\|^2 \mid \hat\lambda_f\Big] &=
2 \sigma^4 \tr\big( \bm{S}_{\hat\lambda_f,q}^4 \big) + 4 \sigma^2\bm{f}^T\big(\bm{I}_n-\bm{S}_{\hat\lambda_f,q} \big)^2\bm{S}_{\hat\lambda_f,q}^2\bm{f} = \\ & = 
\Big\{2 \sigma^4 \tr\big( \bm{S}_{\lambda_f,q}^4 \big) + 4 \sigma^2\bm{f}^T\big(\bm{I}_n-\bm{S}_{\lambda_f,q} \big)^2\bm{S}_{\lambda_f,q}^2\bm{f}\Big\}\Big\{1+O\Big(\frac{\hat\lambda_f}{\lambda_f}-1\Big)\Big\},
\end{align*}
by Lemma~\ref{lemma:trace}, and by Lemma~\ref{lemma:quad_trace_at_oracle}.
By Lemma~\ref{lemma:variance} we conclude that
\begin{align*}
\mathbb{E}\Big\{\mathbb{E}\Big[\|\bm{\hat f}_{\hat\lambda_f,q}-\bm{f}\|^2 \mid \hat\lambda_f\Big]\Big\} &= \frac1n
\Big\{\sigma^2 \tr\big( \bm{S}_{\lambda_f,q}^2 \big) + \bm{f}^T\big(\bm{I}_n-\bm{S}_{\lambda_f,q}\big)^2\bm{f}\Big\}\{1+o(1)\},\\
\mathbb{V}\Big\{\mathbb{E}\Big[\|\bm{\hat f}_{\hat\lambda_f,q}-\bm{f}\|^2 \mid \hat\lambda_f\Big]\Big\} &= \frac1{n^2}
o\Big\{\sigma^2 \tr\big( \bm{S}_{\lambda_f,q}^2 \big) + \bm{f}^T\big(\bm{I}_n-\bm{S}_{\lambda_f,q}\big)^2\bm{f}\Big\}^2,\\
\mathbb{E}\Big\{\mathbb{V}\Big[\|\bm{\hat f}_{\hat\lambda_f,q}-\bm{f}\|^2 \mid \hat\lambda_f\Big]\Big\} &= \frac1{n^2}
\Big\{2 \sigma^4 \tr\big( \bm{S}_{\lambda_f,q}^4 \big) + 4 \sigma^2\bm{f}^T\big(\bm{I}_n-\bm{S}_{\lambda_f,q} \big)^2\bm{S}_{\lambda_f,q}^2\bm{f}\Big\}\{1+o(1)\},
\end{align*}
so that again by Lemma~\ref{lemma:trace}, and by Lemma~\ref{lemma:quad_trace_at_oracle},
\[
\mathbb{V}\Big(\|\bm{\hat f}_{\hat\lambda_f,q}-\bm{f}\|^2\Big) =
o\Big\{\mathbb{E}\Big(\|\bm{\hat f}_{\hat\lambda_f,q}-\bm{f}\|^2\Big)\Big\}^2.
\]

Now, if $\mathbb{E}\Big(\|\bm{\hat f}_{\hat\lambda_f,q}-\bm{f}\|^2 \Big)> C\sigma^2 r_n(\lambda_\beta,\beta)^2$, $C>1$, then by Chebyshev's inequality,
\begin{align*} 
\mathbb{P}\Big( \bm{f} \in \hat{\mathcal{D}}_n \Big) &=
\mathbb{P}\Big\{ \|\bm{\hat f}_{\hat\lambda_f,q}-\bm{f}\|^2 \le \sigma^2 r_n(\lambda_\beta,\beta)^2 \Big\} \le\\
&\le
\bigg[1 + \Big\{ \mathbb{E}\Big(\|\bm{\hat f}_{\hat\lambda_f,q}-\bm{f}\|^2 \Big) - \sigma^2\, r_n(\lambda_\beta, \beta)^2 \Big\}^2\big/
\mathbb{V}\Big(\|\bm{\hat f}_{\hat\lambda_f,q}-\bm{f}\|^2 \Big)\bigg]^{-1} = o(1).
\end{align*}

From Sections~\ref{appendix:coverage:check_S1} 
, $r_n(\lambda,q)^2 = \kappa_q(0,1)\lambda^{-1/(2q)}/n\{1+o(1)\}$.
By the definition of $\lambda_f$ (cf.~\citealp{Krivobokova:2013}),
\begin{align*}&
\mathbb{E}\Big(\|\bm{\hat f}_{\hat\lambda_f,q}-\bm{f}\|^2\Big) \ge
\frac{\sigma^2}n \tr\big( \bm{S}_{\lambda_f,q}^2 \big) + \frac1n \bm{f}^T\big(\bm{I}_n-\bm{S}_{\lambda_f,q}\big)^2\bm{S}_{\lambda_f,q}\bm{f}\\ & \quad=
\frac{\sigma^2}n \Big\{2 \tr\big( \bm{S}_{\lambda_f,q}^2 \big) - \tr\big( \bm{S}_{\lambda_f,q}^3 \big)\Big\}\{1+o(1)\} =
\frac{\sigma^2}n\Big\{ 2\kappa_q(0,2)-\kappa_q(0,3) \Big\}\lambda_f^{-1/(2q)}\{1+o(1)\}.
\end{align*}
From~\eqref{eq:Bayes_oracle} and~\eqref{eq:GCV_oracle}, 
\[
\lambda_{\beta}^{-1/(2\beta)} \ge \big\{\kappa_q(1,2)\big/\kappa_{\beta}(0,2)\big\}^{1/(2\beta+1)}\lambda_f^{-1/(2q)}\{1+o(1)\}.
\]
By basic properties of the gamma function, for any $\beta>1/2$,
\[
2\kappa_q(0,2)-\kappa_q(0,3) > \kappa_\beta(0,1) \big\{\kappa_q(1,2)\big/\kappa_\beta(0,2)\big\}^{1/(2\beta+1)},
\quad q\in[\beta/2,\,\beta],
\]
so that the probability that $\bm{f}$ belongs to $\hat{\mathcal{D}}_n$ is indeed $o(1)$.

\bibliographystyle{chicago}
{\small \bibliography{empirical_smoothing}}

\end{document}